\documentclass[10pt,leqno]{siamltex}
\usepackage{amsfonts,amsmath,amssymb}
\usepackage{graphics,url,color,epsfig}
\usepackage[hidelinks]{hyperref}
\usepackage{stmaryrd}
\usepackage{enumitem}
\usepackage{algorithm}

\newtheorem{example}{Example}[section]
\newtheorem{remark}{Remark}[section]

\newcommand{\px}[1][x]{\partial_{#1}}
\newcommand{\dx}[1][x]{\,{\rm d}#1}

\newcommand{\cpt}[2][t]{{}_CD^{#2}_{0,#1}}

\newcommand{\mfrac}[1][2]{\frac{1}{2}}

\newcommand{\hRule}{\rule{0.95\linewidth}{0.2mm}}

 \allowdisplaybreaks

\title{A new class of semi-implicit methods
with linear complexity for nonlinear fractional differential equations
\thanks{This work was supported by ARC Discovery Project DP150103675 and
 the MURI/ARO on ``Fractional PDEs for Conservation Laws and Beyond: Theory, Numerics and  Applications  (W911NF-15-1-0562)''.}
}

\author{Fanhai Zeng\thanks{School  of  Mathematical   Sciences,
Queensland   University  of  Technology,   Brisbane,   QLD 4001, Australia
(f2.zeng@qut.edu.au).}
\and Ian Turner$^{\dag,}$\thanks{Australian Research Council Centre of Excellence for Mathematical and Statistical Frontiers, Queensland University of
Technology, Brisbane, QLD 4001, Australia (i.turner@qut.edu.au).}
\and Kevin Burrage$^{\dag,}$\thanks{Visiting Professor, Department of  Computer  Science,     University  of
Oxford, OXI 3QD, UK (kevin.burrage@qut.edu.au).}
\and George Em Karniadakis\thanks{Division of Applied Mathematics, Brown University, Providence RI, 02912
(george\_karniadakis@brown.edu).}
}

\begin{document}

\maketitle

\begin{abstract}
We propose a new class of semi-implicit methods for solving nonlinear fractional differential equations and study their stability. Several versions of our new schemes are proved to be unconditionally
stable by choosing suitable parameters. Subsequently, we develop an efficient strategy to calculate
the discrete
convolution  {for} the approximation of the fractional operator in the semi-implicit method
and we derive an error bound of the fast convolution.
The memory requirement and computational cost
of the present semi-implicit methods with a fast convolution are about
$O(N\log n_T)$  and $O(Nn_T\log n_T)$, respectively, where
$N$ is a suitable positive integer and $n_T$ is the final number of time steps.
Numerical simulations, including the solution of a system of two nonlinear fractional diffusion equations
with different fractional orders in two-dimensions, are  presented to verify the effectiveness of the semi-implicit methods.
\end{abstract}



\section{Introduction}\label{sec1}
Anomalous diffusion equations have attracted considerable interest over the last decade because of their ability to model  transport dynamics in complex systems  \cite{MetKla00,Pod-B99}. In these models, the underlying
fractional differential equations (FDEs) are \emph{nonlocal} and time-dependent,
which may cause computational difficulties due to the nonlocality and singularity of the fractional operator
\cite{Diethelm-B10,BafHes17b,SchLopLub06,YuPK16}.

The aim of this paper is to develop a class of semi-implicit and fast
time-stepping methods to efficiently solve nonlinear time-fractional FDEs. We also focus on how to
reduce the memory requirement and the computational cost of the numerical methods resulting from
the \emph{nonlinearity} and \emph{nonlocality} of  nonlinear time-dependent FDEs.
{ The  \emph{singularity} of the solution of the FDE is dealt with through the use of   correction terms,
a topic that is not investigated in detail in this work.}
The interested reader is referred to \cite{DieFord06,Lub86,ZengZK17} for further details.

Semi-implicit methods have been widely applied to solve
nonlinear integer-order differential equations \cite{AscherRW95,Pareschi2005,KarIO91},
as these methods are efficient and have larger stability regions
than explicit methods, giving rise to a linear system of equations to obtain the numerical solutions.
However, semi-implicit methods for nonlinear time-fractional FDEs
have not been fully addressed in the literature.


There have been some explicit and semi-implicit methods for solving fractional differential equations
proposed recently.
Garrappa and his collaborator
have studied  explicit/semi-implicit methods for fractional ordinary differential equations (FODEs),
and also investigated the stability region of these methods,
see e.g. \cite{GalGar08,GalGar09,Gar09,Gar10,Garrappa15}.
Yuste and  Acedo \cite{YusteAcedo05} proposed an explicit difference method
for linear fractional diffusion equations and a weighted average version of \cite{YusteAcedo05}
was proposed in \cite{Yuste06}.
In 2015, Cao et al. \cite{CaoZZK16} proposed a time-splitting method for
nonlinear FODEs based on linear interpolation, and the stability
of the method was studied numerically. Two implicit-explicit time-stepping methods,
one conditionally stable and the other  unconditionally stable, were proposed
in \cite{CaoZK15}, and   stability analysis was given by following the idea in \cite{Lub86b}.
The stability region of a predictor-corrector method \cite{DieFF04} was investigated
in \cite{Gar10}.  The objective of this work
is to present a new class of semi-implicit methods for nonlinear FODEs and apply them to solve
nonlinear time-fractional FPDEs. The new semi-implicit method yields a linear system with a constant coefficient matrix
when it is applied to nonlinear time-fractional  differential equations.



A computational difficulty of numerical methods for time-dependent FDEs is
caused by the \emph{nonlocality} of the fractional operator \cite{BafHes17b,WangBasu12}.
The direct time-stepping method for the time-fractional operator
$k_{-\alpha}*u(t)=\int_0^tk_{-\alpha}(t-s)u(s)\dx[s]$
yields the discrete convolution as
\begin{equation}\label{dis-conv}
\sum_{k=0}^n\omega_{n-k}u_k,\qquad 0\leq n \leq n_T,
\end{equation}
which requires $O(n_T)$ active memory and $O(n_T^2)$ operations. For the case
$k_{\alpha}(t)=t^{\alpha}/\Gamma(\alpha)$, the convolution $k_{\alpha}*u(t)$ defines
the fractional integral of order $\alpha$ for $\alpha>0$
(or Riemann--Liouville fractional derivative of order $-\alpha$ for $\alpha<0$), see
\cite{SamKilM93,ZengTB2017}.
The direct calculation of \eqref{dis-conv} becomes computationally expensive when
it is applied to discretize the time variable  of high-dimensional time-fractional PDEs;
see \cite{YuPK16}.
Some progress on reducing the memory requirement and computational cost for calculating
\eqref{dis-conv} has been  made, and
we refer the reader to \cite{LubSch02,LopLubSch08,JingLi10,JiangZZZ16,McLean12,BafHes17b,YuPK16,ZengTB2017},
where the coefficients $\omega_n$ are obtained from interpolations
and the kernel function $t^{\alpha-1}/\Gamma(\alpha)$ in the fractional operator is approximated by
a sum-of-exponentials.
In the semi-implicit methods analyzed in this paper, the coefficients $\omega_n$ used in
\eqref{dis-conv} are obtained from the generating functions (see \eqref{fbdf} and \eqref{gngf};
and more generating functions can be found in \cite{Lub86}). Therefore,
the fast methods in \cite{LubSch02,LopLubSch08,JingLi10,JiangZZZ16,McLean12,BafHes17b,YuPK16,ZengTB2017}
are  difficult to be applied here.
In \cite{SchLopLub06}, a fast calculation of \eqref{dis-conv}
was developed, where $\omega_n$ is computed from generating functions that correspond
to the fractional backward difference formula (see \eqref{fbdf}) or   implicit Runge--Kutta methods.
A key idea is to re-express the weight $\omega_n$ in \eqref{dis-conv}  as a
  contour integral  of the form
\begin{equation}\label{eq:omega-n}
\omega_{n} = \frac{\tau}{2\pi i}\int_{\mathcal{C}}e_n(\lambda\tau)F_{\omega}(\lambda)\dx[\lambda],
\end{equation}
where $\tau$ is a time step size, and $e_n$ and $F_{\omega}$ depend on the specific discrete convolution
for the approximation of the corresponding integral operator, see Section \ref{sec2} for more details.
In  \cite{SchLopLub06},
the trapezoidal
rules based on the Talbot contour and the hyperbolic contour were proposed to approximate \eqref{eq:omega-n}.
Recently, Banjai et al. \cite{BanLopSch17}
have carefully analysed the Runge--Kutta-based quadrature method and
extended  the fast method proposed in \cite{SchLopLub06} to solve
 linear hyperbolic problems.  An equally important  goal of this paper is to develop a
fast algorithm to  calculate the discrete convolution \eqref{dis-conv} appearing
in the proposed semi-implicit methods.

The main contributions of this work are briefly summarised  below.
\begin{itemize}[leftmargin=*]
  \item i) A new class of semi-implicit methods for nonlinear FODEs is developed.
  The stability of the present semi-implicit methods is investigated, and
  the stability criteria are given and numerically verified. Several  cases
  of the  presented semi-implicit methods are proved to be unconditionally stable and
  are verified by numerical  simulations; see Theorems \ref{thm:4-1}--\ref{thm:4-2},
  Table \ref{s4:stability-interval}, and  Figure  \ref{eg51fig2}.
  \item ii) We reformulate  the fast convolution in \cite{SchLopLub06}
  to calculate the discrete convolutions to  approximate the fractional operator.
  This modification  makes the present fast method  much   easier
  to  calculate for a wider class of discrete convolutions given by \eqref{dis-conv}
  with the coefficients $\omega_n$ defined by \eqref{fbdf} or \eqref{gngf};
  see  \cite{Lub86} for other choices of $\omega_n$.
  An error bound of the   fast convolution is obtained, depending only on the
discretization error of the contour integral,
see \eqref{eq:err}, Table \ref{s5:tb5},   and  Figure \ref{eg31fig4}.
The most important difference between our fast convolution approach and \cite{SchLopLub06} is that
a series of auxiliary ODEs are   solved  by using the \emph{backward Euler} method,
while the ODEs in \cite{SchLopLub06} were solved by a multi-step method
or an implicit Runge--Kutta method according to the specific discrete convolutions
to  approximate the fractional operator; see also \cite{BanLopSch17}.
\end{itemize}

%
We note that the use of the \emph{backward Euler} method to solve the auxiliary ODEs attributes
no additional errors to the whole truncation error of  our fast method as was shown in \cite{BanLopSch17,SchLopLub06},
which is verified by numerical simulations;
see \eqref{eq:err}, Figure \ref{eg31fig4}, and Table \ref{s5:tb5}.
For the fast method based on interpolation given in \cite{LubSch02,LopLubSch08,JingLi10,JiangZZZ16,McLean12,BafHes17b,YuPK16,ZengTB2017},
the  auxiliary ODEs  can be solved exactly, see the detailed
implementation in \cite{ZengTB2017}.

We present several numerical simulations  to verify the accuracy and efficiency
of the  fast method,
demonstrating significant savings in memory and cost, especially when it is applied to solve high-dimensional
time-fractional PDEs, see Example \ref{s5-eg-2}.


This paper is organized as follows.
A new class of semi-implicit methods for nonlinear FODEs is proposed in Section \ref{sec:IMEX}, and
the linear stability of these methods is also analyzed.
A new fast implementation of the semi-implicit methods is presented in Section \ref{sec2}.
Numerical simulations are given to verify the effectiveness of the
semi-implicit and fast methods in Section \ref{sec:numerical}
before the conclusion is given in the last section.

\section{Semi-implicit time-stepping methods}\label{sec:IMEX}
Consider the following  nonlinear FODE
\begin{equation}\label{s4:fiv}
{}_{C}D^{\alpha}_{0,t}u(t) = \lambda u(t) +  f(u(t),t),{\quad}u(0)=u_0,{\quad}t\in(0, T],
\end{equation}
where $0<\alpha\leq 1$, $\lambda\in \mathbb{C}, Re(\lambda)<0$,
and $f(u,t)$ is a nonlinear function with respect to $u$,
and ${}_{C}D^{\alpha}_{0,t}$ is the Caputo fractional derivative operator defined by
\begin{equation}\label{eq:cpt}
\cpt{\alpha}{u(t)}
=\frac{1}{\Gamma(1-\alpha)}\int_{0}^t(t-s)^{-\alpha}u'(s)\dx[s].
\end{equation}

We also assume that the solution $u(t)$ to \eqref{s4:fiv}  satisfies
\begin{equation}\label{s4:solu}
u(t)-u(0) = \sum_{n=1}^{m}c_nt^{\sigma_n}
+ t^{\sigma_{m+1}}\tilde{u}(t),{\quad}0<\sigma_n<\sigma_{n+1},
\end{equation}
where $\tilde{u}(t)$ is uniformly bounded for $t\in[0,T]$.
The above assumption holds in real applications, see, for example,
\cite{Diethelm-B10,FordMorReb13,Luchko11,Lub83,Pod-B99},
 {in which $\sigma_n\in\{i+j\alpha,i,j\in Z^+\}$ if $f(u(t),t)$ is sufficiently smooth  for $t\in[0,T]$}.
\subsection{Derivation of the semi-implicit methods}

Denote by $t_j=j\tau\,(j\geq 0)$   the grid points, where $\tau=T/n_T$ is the stepsize, and $n_T$
is a positive integer. Let $u_n=u(t_n)$ and denote
\begin{equation}\begin{aligned}\label{s4:Dalf}
D_{\tau}^{(\alpha,n,m,\sigma)}u=\frac{1}{\tau^{\alpha}}\sum_{j=0}^{n}\omega^{(\alpha)}_{n-j}(u_j-u_0)
+\frac{1}{\tau^{\alpha}}\sum_{j=1}^{m}w^{(\alpha)}_{n,j}(u_j-u_0),
\end{aligned}\end{equation}
{ where  the  quadrature weights $\omega^{(\alpha)}_{j}$  satisfy the following generating functions \cite{Lub86}
\begin{eqnarray}
\omega(p,\alpha,\tau,z)&=& \left(\frac{1}{\tau}\sum_{k=1}^p\frac{1}{k}(1-z)^k\right)^{\alpha}
=\sum_{n=0}^{\infty}\omega_{n}z^n=\frac{1}{\tau^{\alpha}}\sum_{n=0}^{\infty}\omega^{(\alpha)}_{n}z^n,\label{fbdf}
\end{eqnarray}
or
\begin{eqnarray}
\omega(p,\alpha,\tau,z)&=&\left(\frac{1-z}{\tau}\right)^{\alpha}
\sum_{k=1}^p g^{(\alpha)}_{k-1}(1-z)^{k-1}
=\sum_{n=0}^{\infty}\omega_{n}z^n=\frac{1}{\tau^{\alpha}}\sum_{n=0}^{\infty}\omega^{(\alpha)}_{n}z^n,\label{gngf}
\end{eqnarray}
in which $g^{(\alpha)}_{p-1}(1\leq p \leq 6)$ are given by (see \cite{GalGar08})
\begin{equation}\label{s2:gk}
\begin{aligned}
g^{(\alpha)}_0=& 1, {\qquad\quad}
g^{(\alpha)}_1=  \frac{\alpha}{2},{\qquad\quad}
g^{(\alpha)}_2=\frac{\alpha^2}{8} + \frac{5\alpha}{24}, \\
g^{(\alpha)}_3=& \frac{\alpha^3}{48} + \frac{5\alpha^2}{48} + \frac{\alpha}{8}, {\qquad\qquad}
g^{(\alpha)}_4= \frac{\alpha^4}{384} +\frac{5\alpha^3}{192} + \frac{97\alpha^2}{1152}
+ \frac{251\alpha}{2880},\\
g^{(\alpha)}_5=&\frac{\alpha^5}{3840} + \frac{5\alpha^4}{1152} + \frac{61\alpha^3}{2304}
 + \frac{401\alpha^2}{5760} + \frac{19\alpha}{288}.
\end{aligned}\end{equation}
We refer   readers to \cite{Lub86} for other choices of generating functions.}

Once the quadrature weights  $\omega^{(\alpha)}_{j}$  are given, the starting weights $w^{(\alpha)}_{n,j}$
in \eqref{s4:Dalf} are chosen such  that
$$\sum_{j=0}^{n}\omega^{(\alpha)}_{n-j}u_j
+ \sum_{j=1}^{m}w^{(\alpha)}_{n,j}u_j
=\frac{\Gamma(\sigma_r+1)}{\Gamma(\sigma_r+1-\alpha)}n^{\sigma_r-\alpha}$$
for some $u(t)=t^{\sigma_r},\, {r=1,2,...,m}$.
We refer readers to \cite{DieFord06,Lub86,ZengZK17} for more information on how to determine
the starting weights and their properties.

Using the relationship  ${}_{C}D^{\alpha}_{0,t}u(t)=k_{-\alpha}*(u-u(0))(t)$ and \eqref{s4:solu},
we can apply the fractional linear multi-step method (FLMM)  \eqref{s4:Dalf}  to  discretize the Caputo
fractional derivative operator in \eqref{s4:fiv}, which yields
\begin{equation}\label{s4:eq-1}
D_{\tau}^{(\alpha,n,m,\sigma)}u = \lambda u_n +  f_n
+ O(\tau^pt_n^{\sigma_{m+1}-p-\alpha})+ O(\tau^{\sigma_{m+1}+1}t_n^{-\alpha-1}),
\end{equation}
where $f_n = f(u_n,t_n)$ and $D_{\tau}^{(\alpha,n,m,\sigma)}$ is defined by \eqref{s4:Dalf}.

Let $U_n$ be the approximate solution of $u(t_n)$.  From \eqref{s4:eq-1},
we derive the following fully implicit method
\begin{equation}\label{s4:IM}
D_{\tau}^{(\alpha,n,m,\sigma)}U = \lambda U_n  +  f(U_n,t_n),
\end{equation}
where $D_{\tau}^{(\alpha,n,m,\sigma)}$ is defined by \eqref{s4:Dalf}.
We present \eqref{s4:IM} in order that we can compare it with the
semi-implicit method developed in the following  section.

Cao et al. \cite{CaoZZK16} proposed two methods to linearize the
nonlinear term $f(u_n,t_n)$ in \eqref{s4:eq-1} (see Eq. (2.31) in \cite{CaoZZK16}).
We list the two linearization approaches below:
\begin{itemize}
  \item Extrapolation with correction terms
\begin{equation}\label{s4:extrapolation}
f_n =  2f_{n-1} -f_{n-2}+ \sum_{j=1}^{m_f} w^{(f)}_{n,j}(f_j - f_0)
+O(\tau^2t_n^{\delta_{m_f+1}-2}),
\end{equation}
where $0<\delta_{r}<\delta_{r+1}$, $\delta_{r}\in \{\sigma_k\}\cup\{\sigma_k-\alpha\}$,
and $w^{(f)}_{n,j}$ are chosen such that
$f_n =  2f_{n-1} -f_{n-2}+ \sum_{j=1}^{m_f} w^{(f)}_{n,j}f_j$ for $f=t^{\delta_k}$,
 $1\leq k \leq m_f$.
  \item Taylor expansion  with corrections
\begin{equation}\label{s4:taylor}
\begin{aligned}
f_n \approx &  f_{n-1} + \tau \px[t]f(u_{n-1},t_{n-1})
+ \sum_{j=1}^{m_1} W^{(1)}_{n,j}(f_j - f_0)\\
&+ \px[u]f(u_{n-1},t_{n-1}) \Big(u_n-u_{n-1}+\sum_{j=1}^{m_2} W^{(2)}_{n,j}(u_j - u_0)\Big),
\end{aligned}\end{equation}
where the starting weights  $W^{(1)}_{n,j}$ and $W^{(2)}_{n,j}$ are
not used in this paper and the interested readers can refer to \cite{CaoZZK16}.
\end{itemize}

Compared with \eqref{s4:taylor}, the first approach \eqref{s4:extrapolation} is much simpler.
We will make a modification of
\eqref{s4:extrapolation} and \eqref{s4:taylor} to derive the new semi-implicit methods.
Here, we mainly focus on a modification of \eqref{s4:extrapolation},
which is presented below
\begin{equation}\label{s4:extrapolation-2nd}
\begin{aligned}
f_n =& f_n - E^{n,m_f,\delta}_q(f)
-\kappa E^{n,m_u,\sigma}_q(u)
+O(\tau^qt_n^{\delta_{m_f+1}-q})+O(\tau^qt_n^{\sigma_{m_u+1}-q}),
\end{aligned}\end{equation}
where $\kappa$ is a  constant that may depend on $\px[u]f(u,t)$,
and $E^{n,m,\sigma}_q(u)$ is given by
\begin{equation}\label{s4:Ep-m}
E^{n,m,\sigma}_q(u)=\begin{aligned}
& E^n_q(u)- \sum_{j=1}^{m}w^{(u)}_{n,j}(u_j - u_0),
\end{aligned}\end{equation}
where  $\{w^{(u)}_{n,j}\}$ are chosen such that
$E^{n,m,\sigma}_q(u)=0$ for
 $u=t^{\sigma_r},r=1,2,\cdots,m$. Here $E^n_q(u)$ is
 a $q$th-order perturbation that is defined by
\begin{equation}\label{s4:Ep}
E^n_q(u)=\left\{\begin{aligned}
& u_n-u_{n-1},&  q = 1,\\
& u_n-2u_{n-1} +u_{n-2},&q=2.
\end{aligned}\right.\end{equation}

We can  interpret $E^n_q(u)$ as a penalty term that balances the stability and
accuracy of the proposed method.
Note that $E^n_q(u)=O(\tau^q)$ if $u(t)$ is a smooth function for $t\in[0,T]$.
Higher-order perturbations $E^n_q(u)=O(\tau^q)$ of order
$ {q\geq 3}$ can be constructed,
which are not investigated here.
In real applications, the solution to the considered FDE is often non-smooth, and therefore
 correction terms $\sum_{j=1}^{m_u} w^{(u)}_{n,j}(u_j - u_0)$ are added in
\eqref{s4:Ep-m}, such that
$E^{n,m,\sigma}_q(u)$
has smaller magnitude than $E^n_q(u)$.

Combining \eqref{s4:eq-1} and \eqref{s4:extrapolation-2nd}, we have
\begin{equation}\begin{aligned}\label{s4:IMEX-0}
D_{\tau}^{(\alpha,n,m,\sigma)}u
=& \lambda u_n +  f_n - E^{n,m_f,\delta}_q(f)
-\kappa E^{n,m_u,\sigma}_q(u)+R^n,
\end{aligned}\end{equation}
where the truncation error $R^n$ satisfies
\begin{equation}\begin{aligned}\label{s4:Rn}
R^n=O(\tau^pt_n^{\sigma_{m+1}-p-\alpha})+O(\tau^{\sigma_{m+1}+1}t_n^{-\alpha-1})
+O(\tau^qt_n^{\delta_{m_f+1}-q})+O(\tau^qt_n^{\sigma_{m_u+1}-q}).
\end{aligned}\end{equation}


From \eqref{s4:IMEX-0}, we obtain the following semi-implicit scheme
\begin{equation} \label{s4:IMEX}\begin{aligned}
D_{\tau}^{(\alpha,n,m,\sigma)}U=
 \lambda U_n +  F_n - E^{n,m_f,\delta}_q(F) -\kappa E^{n,m_u,\sigma}_q (U),
\end{aligned} \end{equation}
where $F_n= f(U_n,t_n)$,  $E^{n,m,\sigma}_q$ is defined by \eqref{s4:Ep-m}, and $D_{\tau}^{(\alpha,n,m,\sigma)}$ is defined by \eqref{s4:Dalf}.

We can also make a modification of \eqref{s4:taylor} to derive another semi-implicit method,
i.e., we just need to replace $\px[u]f(u_{n-1},t_{n-1}) u_n$ in \eqref{s4:taylor} with
$$-\kappa u_n + (\kappa+\px[u]f(u_{n-1},t_{n-1})) \big(u_n-E^{n,m_u,\sigma}_q (u)\big).$$
From \eqref{s4:extrapolation-2nd} and the above expression, we obtain the second semi-implicit
method
\begin{equation}\begin{aligned}\label{s4:IMEX2}
D_{\tau}^{(\alpha,n,m,\sigma)}&U= (\lambda-\kappa) U_n +  F_{n-1}  + \tau \px[t]f(U_{n-1},t_{n-1})
+ \sum_{j=1}^{m_1} W^{(1)}_{n,j}(F_j - F_0)\\
&+ \px[u]f(U_{n-1},t_{n-1}) \Big(-U_{n-1}+\sum_{j=1}^{m_2} W^{(2)}_{n,j}(U_j - U_0)\Big)\\
&+\big(\kappa+\px[u]f(U_{n-1},t_{n-1})\big) \Big(U_n-E^{n,m_u,\sigma}_q (U)\Big).
\end{aligned}\end{equation}

If all the correction terms in \eqref{s4:IMEX} and \eqref{s4:IMEX2} are omitted,
both modified methods  reduce to the same method for a linear problem for $q=2$.
Compared with the first method \eqref{s4:IMEX},
the second method  \eqref{s4:IMEX2} seems much more complicated and
involves using partial derivatives of the nonlinear term $f(u,t)$.
In the remaining sections of  this paper, we mainly focus on the theoretical analysis of
the first method \eqref{s4:IMEX} and its application to solving nonlinear FDEs.
\subsection{Linear stability}
We investigate the linear stability of the method \eqref{s4:IMEX} in this subsection.
Let $f(u,t) = \rho u$.  For simplicity, we first drop
all the correction terms since they do not affect the stability of
the proposed method under some suitable conditions. In such a case,   \eqref{s4:IMEX} becomes
\begin{equation}\begin{aligned}\label{s4:IMEX-3}
\frac{1}{\tau^{\alpha}}\sum_{j=0}^{n}\omega^{(\alpha)}_{n-j}(U_j-U_0)
= (\lambda +\rho)U_n - (\rho+\kappa) E_q^n(U),
\end{aligned}\end{equation}
where $\omega^{(\alpha)}_{n}$ satisfies
$\omega(p,\alpha,1,z)=\sum_{n=0}^{\infty}\omega^{(\alpha)}_{n}z^n$, and
$\omega(p,\alpha,\tau,z)$ is defined by \eqref{fbdf} or \eqref{gngf}.

In the following, we  analyze the stability of \eqref{s4:IMEX-3} for $q=2$.
 Eq.  \eqref{s4:IMEX-3}  becomes
\begin{equation}\label{s4:IMEX-2}
\frac{1}{\tau^{\alpha}}\sum_{j=0}^{n}\omega^{(\alpha)}_{n-j}(U_j-U_0)
= (\lambda -\kappa)U_n + (\rho + \kappa)(2U_{n-1}-U_{n-2}).
\end{equation}
For $q=1$, we need only to replace $(\rho + \kappa)(2U_{n-1}-U_{n-2})$
in \eqref{s4:IMEX-2} with $(\rho + \kappa)U_{n-1}$.

Define
\begin{equation}\begin{aligned}\label{s4:stability-p}
\mathbb{S}= \mathbb{C}\setminus
\left\{\xi\big|\xi^{\alpha} = \frac{\omega(p,\alpha,1,z)}{(\lambda+\rho)-(\rho+\kappa)(1-z)^q},
|z|\leq 1\right\} ,
\end{aligned}\end{equation}
where  $\omega(p,\alpha,\tau,z)$ is defined by \eqref{fbdf} or \eqref{gngf}.

According to \cite{CaoZZK16,Lub86b}, we have the following two theorems, the proofs of which are given
in Appendix \ref{appenix-A}.
\begin{theorem}\label{thm:4-1}
If $\tau\in\mathbb{S}$ with   $\mathbb{S}$   defined by \eqref{s4:stability-p},
then  method \eqref{s4:IMEX-3} for the model problem
${}_{C}D^{\alpha}_{0,t}u(t) = (\lambda +\rho)u(t)$ is stable.
\end{theorem}
{
\begin{theorem}\label{thm:4-1-b}
If $\tau\in\mathbb{S}$ with   $\mathbb{S}$   defined by \eqref{s4:stability-p},
$\sigma_m<\alpha+p$,   $\sigma_{m_u},\delta_{m_f}<q$,
then  method \eqref{s4:IMEX} for the model problem
${}_{C}D^{\alpha}_{0,t}u(t) = (\lambda +\rho)u(t)$ is stable.
\end{theorem}
}


\begin{remark}
Let $\alpha>0$  be a rational number, i.e., $\alpha=n/r$, $r,n$ are positive integers
and $\mathrm{gcd}(r,n)=1$. Define the polynomial
$$P(\hat{z})={\hat{z}}^n\sum_{k=1}^p g^{(\alpha)}_{k-1}{\hat{z}}^{r(k-1)}
+\tau^{\alpha}(\rho+\kappa)\hat{z}^{qr}-\tau^{\alpha}(\lambda+\rho),
{\quad}\hat{z}=(1-z)^{1/r},{\quad}|z|\leq 1,$$
where $g^{(\alpha)}_{k}$ are defined by \eqref{s2:gk}.
If the generating function \eqref{gngf} is used and
$P(\hat{z})$ has no root in the domain defined by
$\hat{z}=(1-z)^{1/r},|z|\leq 1$,  then the method \eqref{s4:IMEX-3} is stable.
\end{remark}

Several special cases of \eqref{s4:stability-p} are given in the following theorem,
the proof of which is provided in Appendix \ref{appenix-B}.
\begin{theorem}\label{thm:4-2}
Suppose that $\lambda<0,\rho\leq 0$.
For $p=1,2$, the method \eqref{s4:IMEX-3} is unconditionally stable  if
  \begin{equation} \label{thm:eq:4-2}
\kappa>\left\{\begin{aligned}
& ({\lambda-\rho})/{2},&{\quad}  q = 1,\\
&  ({\lambda-3\rho})/{4},&{\quad} q=2.
\end{aligned}\right.\end{equation}
\end{theorem}

In our numerical simulations, we will apply the second-order generalized
Newton-Gregory formula to discretize the fractional derivative operators
in the considered FDEs, i.e., $\omega(p,\alpha,\tau,z)$
is chosen as
$\tau^{\alpha}\omega(p,\alpha,\tau,z)=(1-z)^{\alpha}(1+\frac{\alpha}{2}-\frac{\alpha}{2}z)$.
Therefore, we first present the stability interval
of the method \eqref{s4:IMEX-3} in Table \ref{s4:stability-interval},
where we set $p=q=2$.
From Table  \ref{s4:stability-interval}, we have the following observations:
\begin{itemize}
  \item For a fixed fractional order $\alpha$, the stability interval increases
  as $\kappa$  increases.
  \item  For a small $\kappa$,
the stability interval may be very small when $\alpha$ tends to zero;
$\kappa=0$ corresponds to the extrapolation method in \cite{CaoZZK16}.
  \item There exists a $\kappa_0\geq 0$, such that the method \eqref{s4:IMEX-3}
  is stable for any $\tau>0$ and $\alpha\in(0,1]$ if $\kappa\geq \kappa_0$. For the case shown
  in Table  \ref{s4:stability-interval}, $\kappa_0 = 1.25$, which verifies \eqref{thm:eq:4-2}.
\end{itemize}
\begin{table}[!h]
\caption{Stability interval of \eqref{s4:IMEX-2} with a second-order generating function
$\omega(p,\alpha,\tau,z) = \tau^{-\alpha}(1-z)^{\alpha}(1+\frac{\alpha}{2}-\frac{\alpha}{2}z)$
for different fractional orders $\alpha$
and $\kappa$, $\lambda=-1,\rho=-2$.}\label{s4:stability-interval}
\centering\footnotesize
\begin{tabular}{|l|l|l|l|l|c|c|c|c|c|c|c|c|}
\hline
 $\kappa$ & $\alpha=0.1$ & $\alpha=0.2$ &  $\alpha=0.5$ &  $\alpha=0.9$    \\
 \hline
$0$    &$(0,5.31\times 10^{-7})$&$(0,1.59\times 10^{-3})$&$(0,1.80\times 10^{-1})$&$(0,6.83\times 10^{-1})$\\
$0.2$  &$(0,3.04\times 10^{-6})$&$(0,3.81\times 10^{-3})$&$(0,2.55\times 10^{-1})$&$(0,8.28\times 10^{-1})$\\
$0.4$  &$(0,2.51\times 10^{-5})$&$(0,1.10\times 10^{-2})$&$(0,3.89\times 10^{-1})$&$(0,1.05\times 10^{0} )$\\
$0.6$  &$(0,3.67\times 10^{-4})$&$(0,4.19\times 10^{-2})$&$(0,6.66\times 10^{-1})$&$(0,1.41\times 10^{0})$\\
$0.8$  &$(0,1.45\times 10^{-2}$ &$(0,2.63\times 10^{-1})$&$(0,1.39\times 10^{0} )$&$(0,2.12\times 10^{0})$\\
$1.0$  &$(0,5.19\times 10^0)$   &$(0,4.98\times 10^{0})$ &$(0,4.50\times 10^{0} )$&$(0,4.08\times 10^{0})$\\
$1.2$  &$(0,5.07\times 10^{7})$ &$(0,1.56\times 10^{4})$ &$(0,1.13\times 10^{2} )$&$(0,2.44\times 10^{1})$\\
$1.24$ &$(0,4.95\times 10^{14})$&$(0,4.86\times 10^{7})$ &$(0,2.81\times 10^{3} )$&$(0,1.46\times 10^{2})$\\
$\textbf{1.25}$ &$(0, \infty)  $ &$(0, \infty)  $&$(0, \infty)  $&$(0, \infty)  $\\
$1.26$ &$(0, \infty)  $&$(0, \infty)  $&$(0, \infty)  $&$(0, \infty)  $\\
$1.40$ &$(0, \infty)  $&$(0, \infty)  $&$(0, \infty)  $&$(0, \infty)  $\\
\hline
\end{tabular}
\end{table}

Next, we plot the stability region of the method \eqref{s4:IMEX-2} under some restrictions.
Let $\rho=\gamma\lambda, \kappa = -\theta\rho=-\theta\gamma\lambda$, and $\xi=\tau^{\alpha}\lambda$.
Then, the stability domain  of the method \eqref{s4:IMEX-3} can be expressed by
\begin{equation}\label{s4:stability-p-2}\begin{aligned}
\mathbb{D}=\mathbb{C}\setminus
\left\{\xi\big| \xi=\frac{\omega(p,\alpha,1,z)}
{(1+\gamma)-\gamma(1-\theta)(1-z)^q},|z|\leq 1\right\}.
\end{aligned}\end{equation}

Figure \ref{fig-stability} displays the stability region defined by \eqref{s4:stability-p-2}
when $Re(\kappa)<Re(\lambda-3\rho)/4$ and
$\omega(2,\alpha,1,z)=(1-z)^{\alpha}(1+\frac{\alpha}{2}-\frac{\alpha}{2}z)$.
We   see that as $\theta$ (or $Re(\kappa)$) increases, the stability region (the shaded area)
becomes larger for a fixed $\alpha$.
\begin{figure}[!h]
\begin{center}
\begin{minipage}{0.4\textwidth}\centering
\epsfig{figure=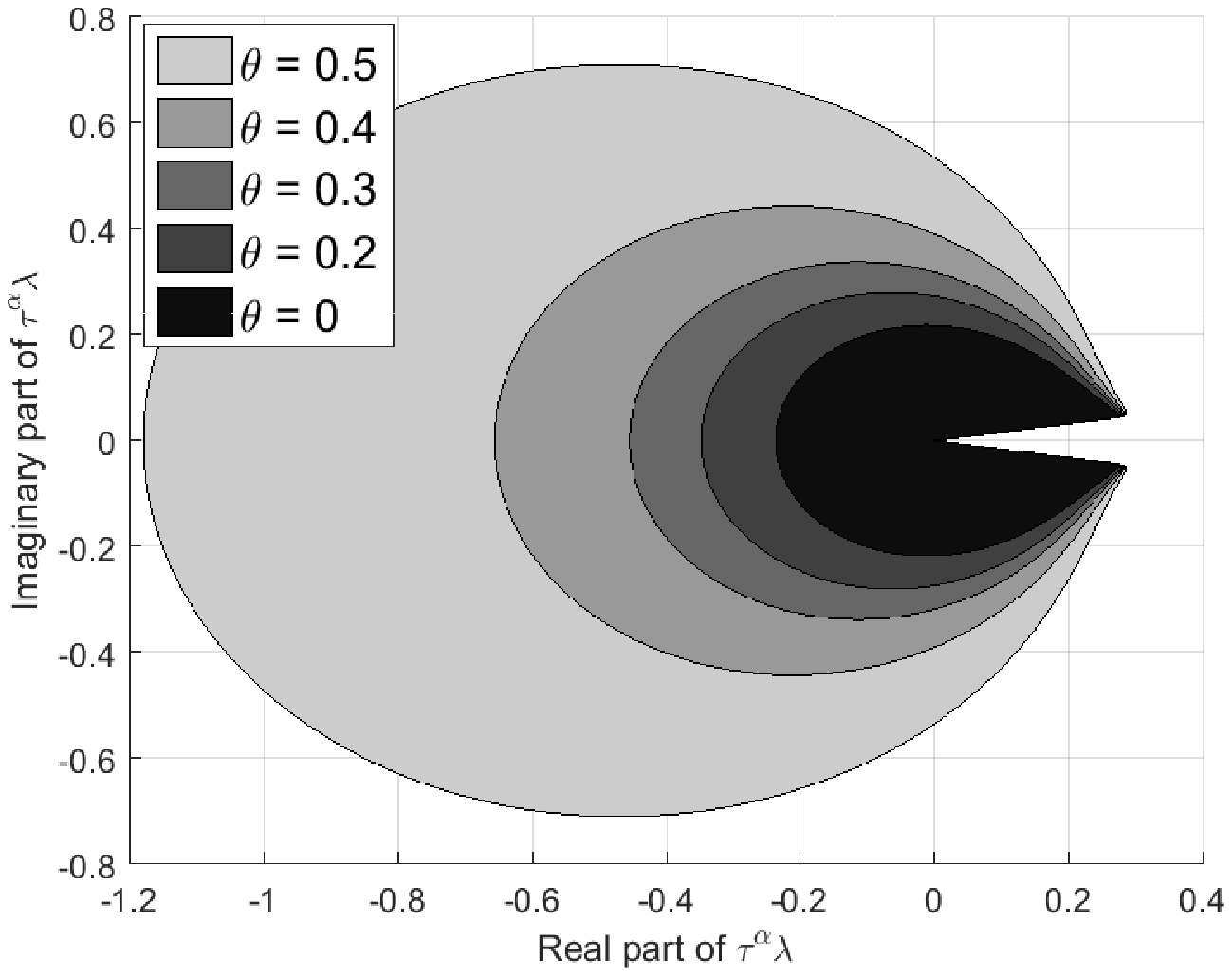,width=5cm} \par {(a) $\alpha=0.1$.}
\end{minipage}
\begin{minipage}{0.4\textwidth}\centering
\epsfig{figure=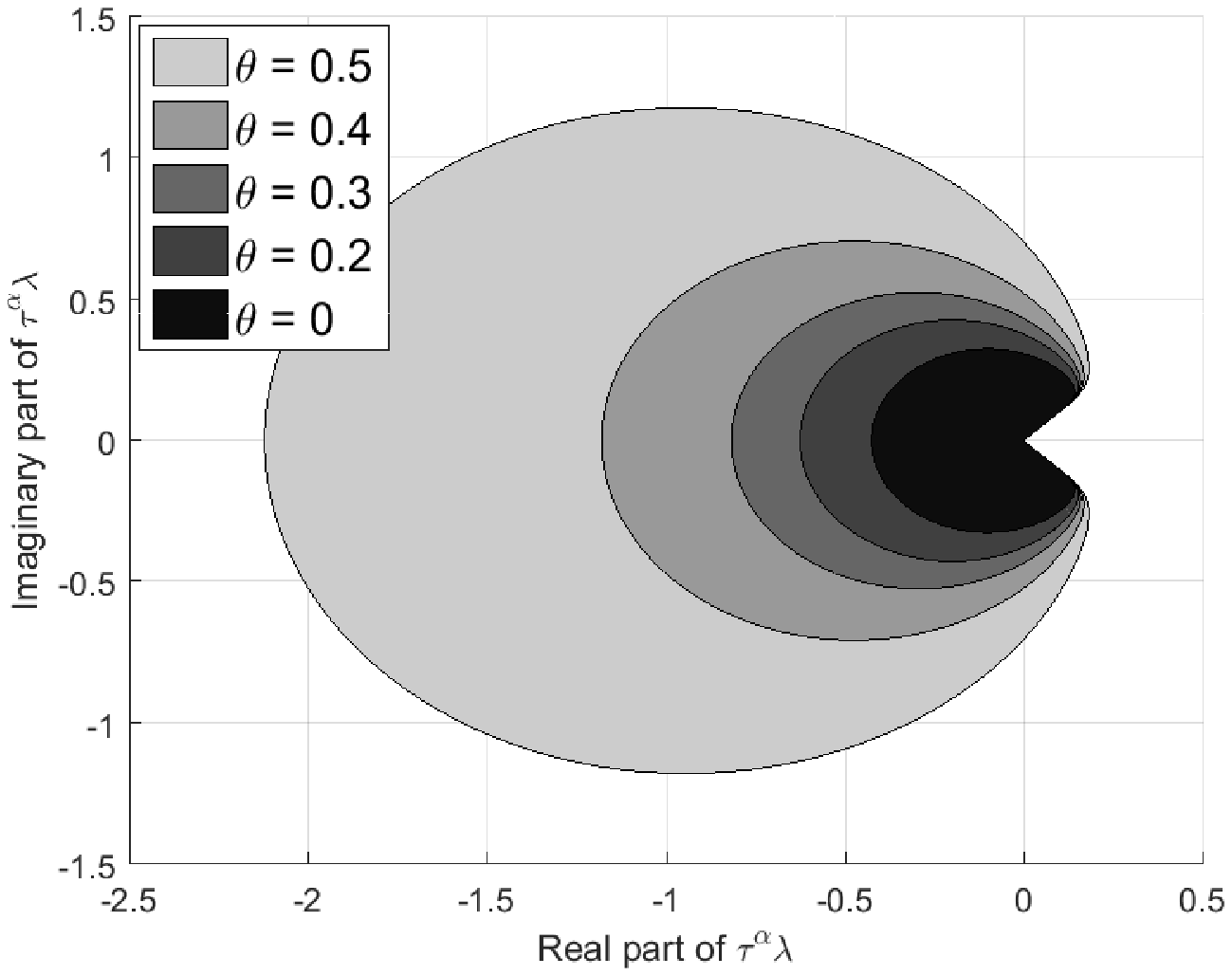,width=5cm} \par {(b) $\alpha=0.5$.}
\end{minipage}\\
\begin{minipage}{0.4\textwidth}\centering
\epsfig{figure=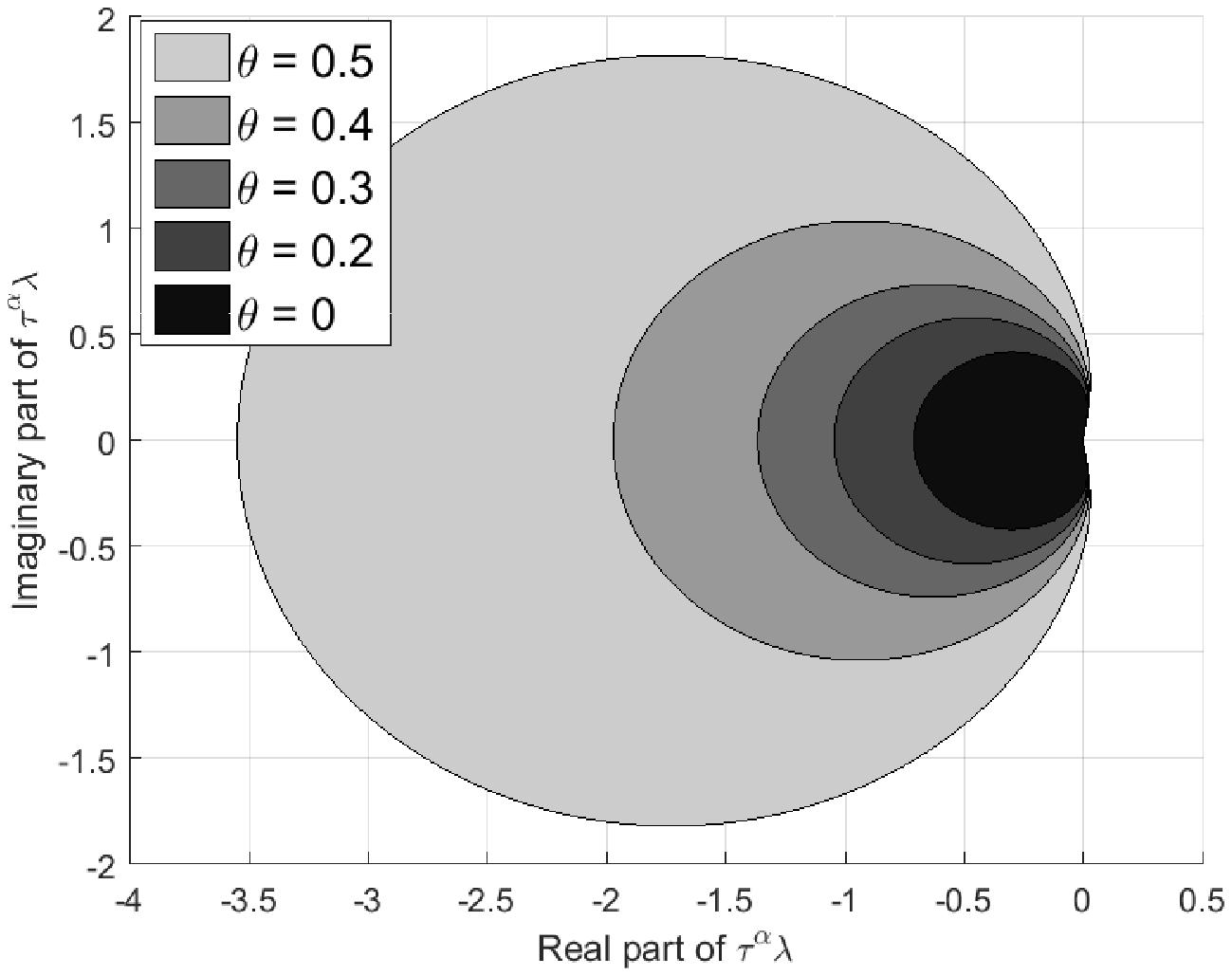,width=5cm} \par {(c) $\alpha=0.9$.}
\end{minipage}
\begin{minipage}{0.4\textwidth}\centering
\epsfig{figure=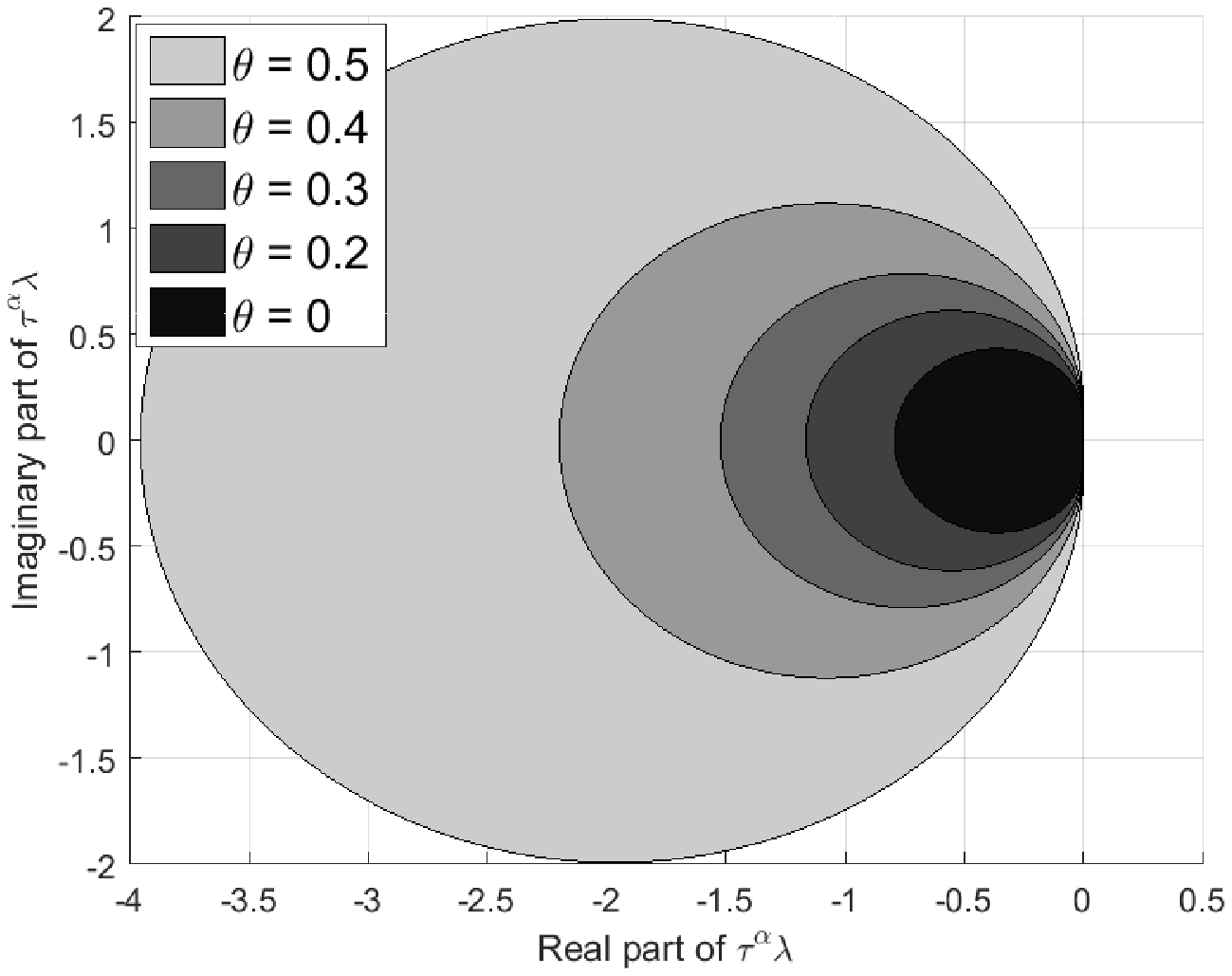,width=5cm} \par {(d)   $\alpha=0.99$.}
\end{minipage}
\end{center}
\caption{Stability region (the shaded area) of the method \eqref{s4:IMEX-3}, $q=2$,
$\rho=2\lambda,\kappa=-\theta\rho$.\label{fig-stability}}
\end{figure}

Figure \ref{fig-stability-2} displays the stability region
when $Re(\kappa)>Re(\lambda-3\rho)/4$. We again observe that the stability region becomes large as
$\theta$ (or $Re(\kappa)$) increases and the stability region contains the whole negative axis, which verifies
Theorem \ref{thm:4-2}.

Figure \ref{fig-stability-4} shows the stability region  of the method \eqref{s4:IMEX-3}
when the  generalized Newton--Gregory formula  of order $p$ is applied, see \eqref{gngf}.
For $\theta=0.5$, we have $Re(\kappa)<Re(\lambda-3\rho)/4$; the method is conditionally stable and
the stability region becomes large as $p$ increases.
For $Re(\kappa)>Re(\lambda-3\rho)/4$, i.e., $\theta =0.63$,
the stability region contains the negative axis, see Figure \ref{fig-stability-5}.
For other fractional orders $\alpha\in (0,1]$, we have similar results.
If the fractional backward difference formula (FBDF) of order $p$ is applied, the stability of the method  \eqref{s4:IMEX-3}
shows similar behavior as shown in Figures \ref{fig-stability-2}--\ref{fig-stability-5}, hence these
results are not shown here.

\begin{figure}[!h]
\begin{center}
\begin{minipage}{0.4\textwidth}\centering
\epsfig{figure=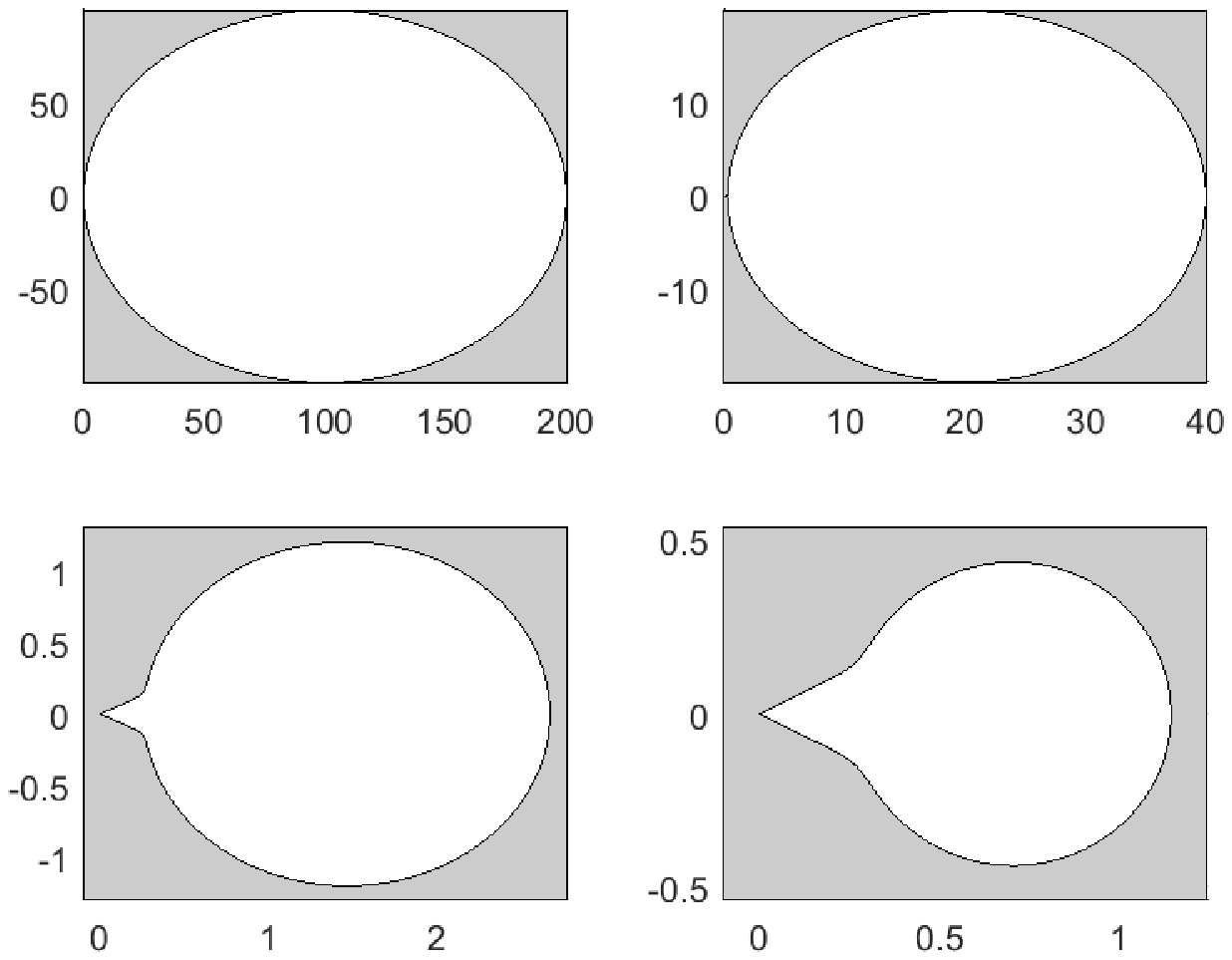,width=5cm} \par {(a) $\alpha=0.3$.}
\end{minipage}
\begin{minipage}{0.4\textwidth}\centering
\epsfig{figure=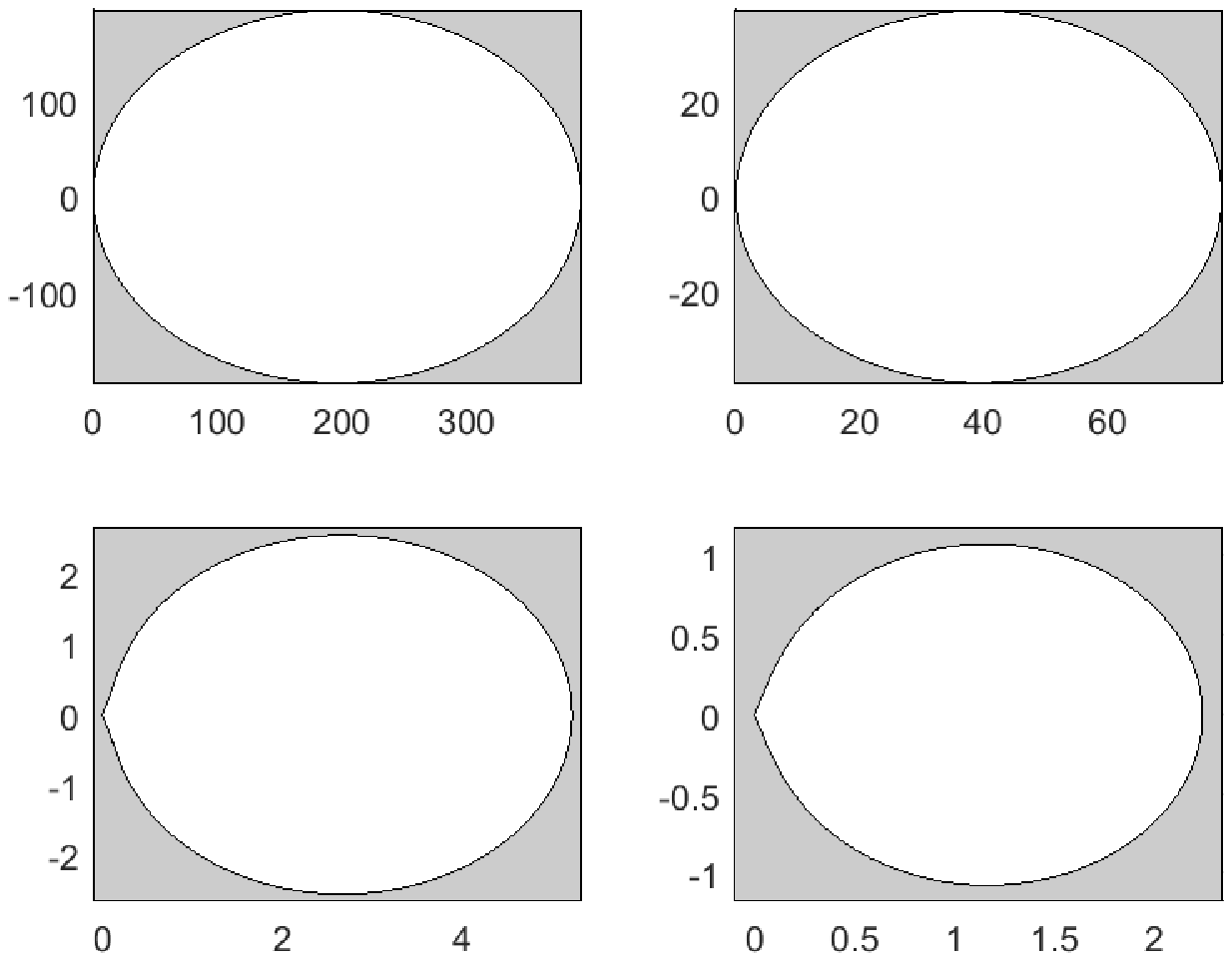,width=5cm} \par {(b) $\alpha=0.8$.}
\end{minipage}
\end{center}
\caption{Stability region (the shaded area)
of the method \eqref{s4:IMEX-3}, $q=2$,
$\rho=2\lambda,\kappa=-\theta\rho$, $\theta=0.626,0.63,0.8,0.7$ in clockwise order
starting from the upper left.\label{fig-stability-2}}
\end{figure}

\begin{figure}[!h]
\begin{center}
\epsfig{figure=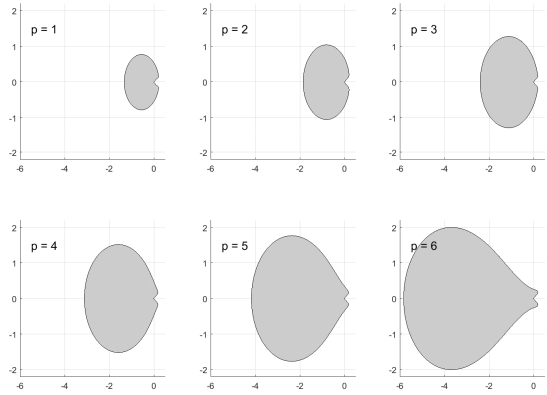,width=9cm}
\end{center}
\caption{Stability region of the method \eqref{s4:IMEX-3} based on generalized Newton--Gregory formula of order $p$, $\alpha=0.4,q=2,\rho=2\lambda,\kappa=-0.5\rho$.
 \label{fig-stability-4}}
\end{figure}

\begin{figure}[!h]
\begin{center}
\epsfig{figure=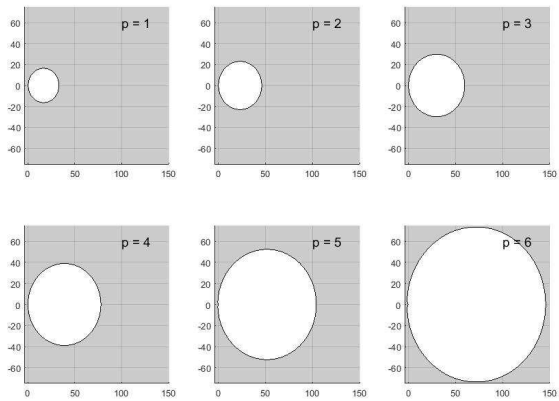,width=9cm}
\end{center}
\caption{Stability region of  \eqref{s4:IMEX-3} based on the FBDF-$p$,
$\alpha=0.4,q=2,\rho=2\lambda,\kappa=-0.63\rho$.
 \label{fig-stability-5}}
\end{figure}

%
%

\subsection{System of equations}
In this subsection, we extend the semi-implicit method \eqref{s4:IMEX} to the following system of equations
\begin{equation}\label{s23:fode}
{}_{C}D^{\vec{\alpha}}_{0,t}\vec{u}(t) = \mathbf{A} \vec{u}(t) +  \vec{f}(\vec{u}(t),t),{\quad}\vec{u}(0)=\vec{u}_0,{\quad}t\in(0, T],
\end{equation}
where $\vec{u}(t)=(u^{(1)}(t),...,u^{(d)}(t))^T$,
$\vec{\alpha}=(\alpha^{(1)},...,\alpha^{(d)})^T$, $\alpha^{(j)}\in(0, 1]$,
$\mathbf{A}\in \mathbb{C}^{d\times d}$,   ${}_{C}D^{\vec{\alpha}}_{0,t}\vec{u}(t)
=({}_{C}D^{{\alpha^{(1)}}}_{0,t}{u^{(1)}}(t),...,
{}_{C}D^{{\alpha^{(d)}}}_{0,t}{u^{(d)}}(t))^T$,
$\vec{f}(\vec{u},t)=(f^{(1)}(\vec{u},t),...,f^{(d)}(\vec{u},t))^T$,
${f}^{(j)}(\vec{u},t)=f^{(j)}(u^{(1)},...,u^{(d)},t)^T$, and $d\in \mathbb{N}$.

Let $\vec{U}_n=(U^{(1)}_n,U^2_n,...,U^{(d)}_n)^T$ be the approximate solution of $\vec{u}(t_n)$.
Similar to \eqref{s4:IMEX}, we directly present the semi-implicit method
for \eqref{s23:fode} as follows:
\begin{eqnarray}
D_{\tau}^{(\vec{\alpha},n,\vec{m},\vec{\sigma})}&&\vec{U}
=\mathbf{A} \vec{U}_n +  \vec{F}_n - E^{n,\vec{m}_f,\vec{\delta}}_q(\vec{F})
-\boldsymbol{\kappa} E^{n,\vec{m}_u,\vec{\sigma}}_q(\vec{U}),\label{s23:fode-2}
 \end{eqnarray}
where $\vec{m}=(m^{(1)},...,m^{(d)})^T$,
$\vec{\sigma}=(\sigma^{(1)},...,\sigma^{(d)})^T$,
$F^{(j)}_n= f^{(j)}(\vec{U}_n,t_n)$,
$D_{\tau}^{(\vec{\alpha},n,\vec{m},\vec{\sigma})}\vec{U}$
$=(D_{\tau}^{(\alpha^{(1)},n,m^{(1)},\sigma^{(1)})}U^{(1)},...,
D_{\tau}^{(\alpha^{(d)},n,m^{(d)},\sigma^{(d)})}U^{(d)})^T$,
$D_{\tau}^{({\alpha^{(j)}},n,{m^{(j)}},{\sigma^{(j)}})}$ is defined by \eqref{s4:Dalf},
$E^{n,\vec{m},\vec{\sigma}}_q(\vec{U})
=(E^{n,{m^{(1)}},\sigma^{(1)}}_q(U^{(1)}),...,E^{n,{m^{(d)}},\sigma^{(d)}}_q(U^{(d)}))^T$,
$E^{n,m,\sigma}_q$ is defined by
\eqref{s4:Ep-m}, and $\boldsymbol{\kappa}$ is a $d\times d$ matrix.



Next, we analyse the linear stability of \eqref{s23:fode-2}.
Let $\vec{f}=\mathbf{B}\vec{u},\mathbf{B}\in \mathbb{C}^{d\times d}$ and omit all the correction
terms in \eqref{s23:fode-2}. Then we obtain
\begin{eqnarray}
D_{\tau}^{(\vec{\alpha},n,\vec{m},\vec{\sigma})}&&\vec{U}
=(\mathbf{A} +\mathbf{B}) \vec{U}_n - (\boldsymbol{\kappa}+\mathbf{B}) E^n_q(\vec{U}).\label{s23:fode-4}
 \end{eqnarray}
Similar to the stability analysis given in the previous subsection, we can   obtain that
the method \eqref{s23:fode-4} is stable if
\begin{equation}\label{s23:fode-3}
\det\left(\diag(\vec{\omega}(p,\vec{\alpha},1,z)) - \diag(\tau^{\vec{\alpha}})
\left[(\mathbf{A} + \mathbf{B})
-  (\mathbf{B} + \boldsymbol{\kappa})(1-z)^q\right]\right)\neq 0,{\quad}|z|\leq 1,
\end{equation}
where $\vec{\omega}(p,\vec{\alpha},1,z))
=(\omega(p,\alpha^{(1)},1,z),..., \omega(p,{\alpha^{(d)}},1,z))^T$,
$\tau^{\vec{\alpha}}=(\tau^{\alpha^{(1)}},...,\tau^{\alpha^{(d)}})^T$, $\sigma^{(i)}_{m^{(i)}}<\alpha^{(i)}+p$,   $\sigma^{(i)}_{m^{(i)}_u},\delta^{(i)}_{m^{(i)}_f}<q$,
$1\leq i \leq d$.

The matrix
$\boldsymbol{\kappa}$ plays an important role in the stability of the method \eqref{s23:fode-2}.
From \eqref{s4:taylor}, we  find that $\boldsymbol{\kappa}$ may depend on the Jacobian of
$\vec{f}(\vec{u},t)$ with respect to $\vec{u}$. It is much more difficult to present an
explicit criterion as shown in Theorem \ref{thm:4-2} for a system equations.
However, we can deduce that the stability of the method \eqref{s23:fode-4} can be enhanced
if $\boldsymbol{\kappa}$ is positive definite for $q=1,2$ with the generating function \eqref{gngf}
for $p=1,2$ applied.
For example,   let $q=1$, $p=1$ or 2, $\vec{U}_0=\vec{0}$, and choose a  matrix
$\boldsymbol{\kappa}$ such that $\boldsymbol{\kappa} +\mathbf{B}$ is symmetric and positive definite.  Then
\begin{eqnarray}
&&\sum_{n=1}^{M}(\vec{U}_n)^TD_{\tau}^{(\vec{\alpha},n,\vec{m},\vec{\sigma})}\vec{U}
-\sum_{n=1}^{M}(\vec{U}_n)^T(\mathbf{A} +\mathbf{B})\vec{U}_n\label{s23:fode-5}\\
=&&-\sum_{n=1}^M(\vec{U}_n)^T(\boldsymbol{\kappa}+\mathbf{B}) E^n_1(\vec{U})
\leq -\frac{1}{2}(\vec{U}_M)^T(\boldsymbol{\kappa}+\mathbf{B})\vec{U}_M.\nonumber
 \end{eqnarray}
If $\mathbf{A} +\mathbf{B}$  is negative definite, then \eqref{s23:fode-5} implies
\begin{equation}\label{s23:fode-6}
2\sum_{n=1}^M(\vec{U}_n)^TD_{\tau}^{(\vec{\alpha},n,\vec{m},\vec{\sigma})}\vec{U}
-2\sum_{n=1}^M(\vec{U}_n)^T(\mathbf{A} +\mathbf{B})\vec{U}_n
+(\vec{U}_M)^T(\boldsymbol{\kappa}+\mathbf{B}) \vec{U}_M
\leq 0.
\end{equation}
Using $\sum_{n=1}^M(\vec{U}_n)^TD_{\tau}^{(\vec{\alpha},n,\vec{m},\vec{\sigma})}\vec{U}\geq 0$
 (see \cite{WangVong14}) yields $\vec{U}_n=\vec{0}$.

A practical option can be $\boldsymbol{\kappa}=\diag(\kappa_1,\kappa_2,...,\kappa_d),\kappa_i\geq0$,
which is used in the numerical simulations. From \eqref{s23:fode-5}, we see that  $\boldsymbol{\kappa}=\diag(\kappa_1,\kappa_2,...,\kappa_d)$ can significantly enhance the stability
of \eqref{s23:fode-4} when $\kappa_i\geq 0$ is sufficiently large for $q=1,2$.
Next, we focus on the other important task of developing a fast method to calculate
the discrete convolutions in the semi-implicit methods.


\section{Fast implementation}\label{sec2}
In this section, we generalize and extend the approach in \cite{SchLopLub06}
to calculate the discrete convolution \eqref{dis-conv}
originating from the discretization of the fractional integral and derivative operators,
where $\omega_{n}$ are the coefficients of the generating function $F_{\omega}(\delta(\xi)/\tau)$,
i.e., $F_{\omega}(\delta(\xi)/\tau)=\sum_{n=0}^{\infty}\omega_{n}\xi^n$,
which has been discussed in the previous section.
In \cite{SchLopLub06},  $\delta(\xi)$ is related to the
linear multistep method for the first-order initial value problem.
In the present fast method, we assume $\delta(\xi)=1-\xi$, i.e., $\delta(\xi)$
corresponds to the backward Euler method.
\subsection{Review of the existing method}
We first recall the fast method proposed in \cite{SchLopLub06}.
The key idea is to re-express the coefficient $\omega_{n}$ in \eqref{dis-conv}
using the integral formula,
and then approximate it using  numerical quadrature.


By Cauchy's integral formula, we have
\begin{equation}\label{CauchyF}
F_{\omega}(\delta(\xi)/\tau)=\sum_{n=0}^{\infty}\omega_{n}\xi^n
=\frac{1}{2\pi i}\int_{\mathcal{C}}
\left(\frac{\delta(\xi)}{\tau}-\lambda\right)^{-1}F_{\omega}(\lambda)\dx[\lambda],
\end{equation}
where $\mathcal{C}$ is a suitable contour.
Define $e_n(z)$ as
\begin{equation} \label{enz}
\left({\delta(\xi)} -z\right)^{-1} = \sum_{n=0}^{\infty}e_{n}(z)\xi^n.
\end{equation}
Then $\omega_n$ can be expressed by
\begin{equation}\label{Flambda}
\omega_{n} = \frac{\tau}{2\pi i}\int_{\mathcal{C}}e_n(\lambda\tau)F_{\omega}(\lambda)\dx[\lambda].
\end{equation}
Inserting \eqref{Flambda} into \eqref{dis-conv} yields
\begin{equation}\label{eq:discrete-conv-2}
D^{(\alpha,n)}_{\tau}u=\sum_{j=0}^{n}\omega_{n-j}u_{j}=\frac{\tau}{2\pi i}
\sum_{j=0}^{n}\int_{\mathcal{C}}e_{n-j}(\lambda\tau)F_{\omega}(\lambda)u_{j}\dx[\lambda].
\end{equation}
When the  integral on the right-hand-side of the above equation
is approximated by a suitable quadrature method, we have the fast convolution developed in \cite{SchLopLub06}.

The fractional backward difference formula (FBDF) and implicit
Runge--Kutta method  were investigated in \cite{SchLopLub06}. For the FBDF of order $p$
(FBDF-$p$), one has $\delta(\xi)=\sum_{k=1}^p\frac{1}{k}(1-\xi)^k$
and $F_{\omega}(\lambda)=\lambda^{\alpha}$, and
the corresponding $e_n(z)$ in \eqref{enz} for $p=1,2$ is given by
\begin{equation} \label{enzfbdf}
e_n(z) =\left\{\begin{aligned}
&(1-z)^{-1-n},\quad \text{FBDF-}1,\\
& \frac{1}{1+2z}\left((2-\sqrt{1+2z})^{-1-n} - (2+\sqrt{1+2z})^{-1-n}\right),
\quad \text{FBDF-}2.
\end{aligned}\right.\end{equation}
It seems much more complicated to obtain $e_n(z)$ for FBDF-$p$ with $p\geq 3$.

Next, we  simplify the method in \cite{SchLopLub06}, that is,
we will always have $e_n(z)=(1-z)^{-1-n}$ used in \eqref{enz}, and
$F_{\omega}(\lambda)$ in \eqref{Flambda} is related to a fractional linear multi-step method (FLMM)
that discretizes the fractional integral or Riemann--Liouville (RL) fractional derivative operator,
which will be given later.

\subsection{A new fast convolution}
We consider \eqref{dis-conv} with the coefficients $\omega_{n}$ derived from
the  FBDF-$p$ ($1\leq p\leq 6$)   for the fractional integral
and RL derivative operators, i.e.,
the coefficients $\omega_{n}$ satisfy \eqref{fbdf}.

Let $\delta(\xi)=1-\xi$ and repeat the procedures \eqref{CauchyF}--\eqref{Flambda}. Then
the coefficient $\omega_{n}$ in \eqref{fbdf} can also be expressed by
\begin{equation}\label{omega-n}
\omega_{n} = \frac{\tau}{2\pi i}\int_{\mathcal{C}}e_n(\lambda\tau)F_{\omega}(\lambda)\dx[\lambda]
= \frac{\tau}{2\pi i}\int_{\mathcal{C}}(1-\lambda\tau)^{-1-n} F_{\omega}(\lambda)\dx[\lambda],
\end{equation}
where   $\mathcal{C}$ is a   contour that surrounds the pole
$\lambda=\frac{1}{\tau}$ of $e_n(\lambda\tau)$
and $F_{\omega}(\lambda)$ is given by
\begin{equation}\label{BDF-F}
F_{\omega}(\lambda)=\omega(p,\alpha,\tau,1-\tau\lambda)
=\lambda^{\alpha}\left(\sum_{k=1}^p\frac{1}{k}(\tau\lambda)^{k-1}\right)^{\alpha}.
\end{equation}

By choosing a suitable contour, \eqref{omega-n} can be approximated with high accuracy.
As was done in \cite{SchLopLub06}, we can apply the trapezoidal rule based on either a  hyperbolic contour or
a Talbot contour to approximate \eqref{omega-n}, which is given by
\begin{equation}\label{omega-n-2}
\omega_{n} \approx \hat{\omega}_{n}=\mathrm{Im}
\bigg( \tau\sum_{k=-N}^{N-1}  w^{(\ell)}_k(1-\lambda^{(\ell)}_k\tau)^{-1-n}
F_{\omega}(\lambda^{(\ell)}_k)\bigg),
\end{equation}
where the quadrature points $\lambda^{(\ell)}_k$ and weights $w^{(\ell)}_k$
 are defined later in this section, see \eqref{weights} or \eqref{weights-2}.
If $\omega_n$ in \eqref{dis-conv} is defined by \eqref{gngf},
then $F_{\omega}(\lambda)$ in \eqref{omega-n} is given by
\begin{equation}\label{GNGF-F}
F_{\omega}(\lambda)=\omega(p,\alpha,\tau,1-\tau\lambda)
=\lambda^{\alpha} \sum_{k=0}^{p-1}g^{(\alpha)}_k(\tau\lambda)^k.
\end{equation}


We now present our  fast  method for calculating
\eqref{dis-conv}  in  Algorithm \ref{fast_conv}.

\begin{algorithm}[!ht]
\begin{minipage}{1\textwidth}
\textbf{Input}: The positive integers $n_0$, $N$, and   $B\geq 2$,
the real number $\alpha$, the time stepsize $\tau$,
the quadrature points $\{{\lambda}^{(\ell)}_k\}$ and weights $\{{w}^{(\ell)}_k\}$
defined by \eqref{weights} or \eqref{weights-2},
the coefficients $\omega_n$  defined by the generating function
$\omega(z)=\omega(p,\alpha,\tau,z)$
(see \eqref{fbdf} or \eqref{gngf}),  and the function
$F_{\omega}(\lambda)$ (see \eqref{BDF-F} or \eqref{GNGF-F}).

\textbf{Output}: ${}_FD^{(\alpha,n)}_{\tau,n_0}u$.\\
\hRule \\[0.1cm]
\begin{itemize}[leftmargin=*]
   \item For each $n\geq n_0$, find   the smallest integer  $L$  satisfying ${n-n_0+1}<2B^L$.\\
  For $\ell=1,2,...,L-1$, determine the  integer $q_{\ell}$ such that
  \begin{equation}\label{b-ell}
  b^{(n)}_{\ell}=q_{\ell}B^{\ell} \quad \text{satisfies}\quad
  {n-n_0+1}-b^{(n)}_{\ell}\in[B^{\ell-1},2B^{\ell}-1].
  \end{equation}
Set $b^{(n)}_{0}={n-n_0}$ and $b^{(n)}_L=0$.
  \item   Decompose the convolution $\sum_{j=0}^{n}\omega_{n-j}u_{j}$ as
  $\sum_{j=0}^{n}\omega_{n-j}u_{j}=\sum_{\ell=0}^{L}u^{(\ell)}_n$,
  where $  u^{(0)}_n = \sum_{j=n-n_0}^{n}\omega_{n-j}u_{j}$ and
  $u^{(\ell)}_n   =\sum_{j=b^{(n)}_{\ell}}^{b^{(n)}_{\ell-1}-1}\omega_{n-j}u_{j}.$

  \item   For every $1\leq \ell\leq L$, approximate $u^{(\ell)}_n$ with
  $\widehat{u}^{(\ell)}_n$, where
  \begin{equation}\label{eq:gauss-3}\begin{aligned}
u^{(\ell)}_n
=&\sum_{j=b^{(n)}_{\ell}}^{b^{(n)}_{\ell-1}-1}\omega_{n-j}u_{j}
=\frac{\tau}{2\pi i}\int_{\Gamma_{\ell}}
(1-\tau\lambda)^{-[n-(b^{(n)}_{\ell-1}-1)]}F_{\omega}(\lambda)y^{(\ell)}(\tau\lambda)\dx[\lambda]\\
\approx&\mathrm{Im}\bigg\{ \sum_{k=-N}^{N-1}{w}_k^{(\ell)}F_{\omega}({\lambda}^{(\ell)}_k)
(1-\tau\lambda_k^{(\ell)})^{-[n-(b^{(n)}_{\ell-1}-1)]}{y}(\tau{\lambda}^{(\ell)}_k)\bigg\}
=\widehat{u}^{(\ell)}_n
\end{aligned} \end{equation}
with $y^{(\ell)}(\tau\lambda)$  given by
$y^{(\ell)}(\tau\lambda)=\tau\sum_{j=b^{(n)}_{\ell}}^{b^{(n)}_{\ell-1}-1}
e_{(b^{(n)}_{\ell-1}-1)-j}(\tau\lambda)u_{j}.$
Here $y^{(\ell)}(\tau\lambda)=y(b^{(n)}_{\ell-1}\tau,b^{(n)}_{\ell}\tau,\tau\lambda)$ is   the
\emph{backward Euler} approximation to the solution at $t=b^{(n)}_{\ell-1}\tau$
of the linear initial-value problem
\begin{equation}\label{ode2}
y'(t)=\lambda y(t) + u(t),{\quad}y(b^{(n)}_{\ell}\tau)=0.
\end{equation}
The quadrature points  ${\lambda}_j^{(\ell)}$ and weights ${w}_j^{(\ell)}$ are
given by \eqref{weights} or \eqref{weights-2}.
\item   Calculate
\begin{equation}\label{s4:Dalf-fast}
{}_FD^{(\alpha,n)}_{\tau,n_0}u=u^{(0)}_n  + \widehat{u}^{(1)}_n+\cdots + \widehat{u}^{(L)}_n
\end{equation}
  with $\widehat{u}^{(\ell)}_n$ defined by \eqref{eq:gauss-3}.
 \end{itemize}
\caption{Fast calculation of $D^{(\alpha,n)}_{\tau}u=D^{(\alpha,n,0,\sigma)}_{\tau}u=\sum_{j=0}^n\omega_{n-j}u_j$,
where $\omega_{n}$ satisfies
$\omega(p,\alpha,\tau,z)=\sum_{n=0}^{\infty}\omega_{n}z^n$, and $\omega(p,\alpha,\tau,z)$ can be
defined by \eqref{fbdf} or \eqref{gngf}.}\label{fast_conv}
\end{minipage}
\end{algorithm}

Our goal below is to determine the quadrature points ${\lambda}_k^{(\ell)}$
and weights ${w}_k^{(\ell)}$ in \eqref{eq:gauss-3},
such that  $\widehat{u}^{(\ell)}_n$ is a good approximation of
$u^{(\ell)}_n$ for any $n\geq n_0$.
The  trapezoidal rule based on
the Talbot contour  \cite{SchLopLub06,Weideman06} or  hyperbolic contour
\cite{LopLubPS2005,SchLopLub06} has been applied to approximate
$\frac{1}{2\pi i}\int_{\Gamma_{\ell}}e_{n-j}(\tau\lambda)F_{\omega}(\lambda)\dx[\lambda]$
in \eqref{eq:gauss-3}, which will be applied in this work.
We present the quadrature points ${\lambda}_k^{(\ell)}$  and weights ${w}_k^{(\ell)}$
used in \eqref{eq:gauss-3}.
\begin{itemize}[leftmargin=*]
  \item
For the trapezoidal rule based on the  optimal  Talbot contour (see \cite{Weideman06}),
the quadrature points ${\lambda}_k^{(\ell)}$  and weights ${w}_k^{(\ell)}$
are given by
\begin{equation}\label{weights}
{\lambda}_k^{(\ell)} = z(\theta_k,N/T_{\ell}),{\quad}
{w}_k^{(\ell)}=\px[\theta]z(\theta_k,N/T_{\ell}), {\quad}\theta_k=\frac{(2k+1)\pi}{2N},
\end{equation}
where
$z(\theta,N)=N\left(-0.4814 +0.6443(\theta\cot(\theta)+i0.5653\theta)\right)$.
Our numerical results indicate that  $N=30$ works well when $T_{\ell}=(2B^{\ell}-2+n_0)\tau$ and
$B$ is not too large.

\item  For the trapezoidal rule based on the hyperbolic contour (see \cite{LopLubPS2005}),
the quadrature points ${\lambda}_k^{(\ell)}$  and weights ${w}_k^{(\ell)}$
are given by
\begin{equation}\label{weights-2}
{\lambda}_k^{(\ell)} = z(\theta_k,\mu_{\ell}),{\quad}
{w}_k^{(\ell)}=\px[\theta]z(\theta_k,\mu_{\ell}),{\quad}
\theta_k=(k+1/2)\hat{h},
\end{equation}
where $z(\theta,\mu_{\ell})=\mu_{\ell}\left(1-\sin(\psi +i\theta)\right)+\sigma$
and $\hat{h}$ is a step length parameter that is chosen as $\hat{h}=\pi/N$ in this paper.
According to \cite{LopLubPS2005}, we can
choose $\sigma=0$, $\psi=0.4\pi$, $\mu_{\ell}=N/(2T_{\ell})$,
$T_{\ell}=(2B^{\ell}-2+n_0)\tau$,
and $N=\lceil-\log(\tau^{1-\alpha}\epsilon)\rceil$ for numerical simulations,
where  $\epsilon$ is a given precision.  {Readers can also refer to \cite{SchLopLub06}, where a strategy was proposed
to choose the parameters for the hyperbolic contour.}
\end{itemize}

Due to the symmetry of the trapezoidal rule,  Eq. \eqref{eq:gauss-3}
can be replaced by
$\widehat{u}^{(\ell)}_n
=\mathrm{Im}\left\{ 2\sum_{k=0}^{N-1}{w}_k^{(\ell)}F_{\omega}({\lambda}^{(\ell)}_k)
(1-\tau\lambda_k^{(\ell)})^{-[n-(b^{(n)}_{\ell-1}-1)]}{y}(\tau{\lambda}^{(\ell)}_k)\right\},$
so that the computational cost can be halved.
The memory requirement and computational cost of the present fast method are about $O(\log n_T)$ and
$O(Nn_T\log n_T)$, respectively, when $n_T$ is suitably large;
see also \cite{LopLubSch08,LubSch02,YuPK16,ZengTB2017}.

{
\begin{remark}
For the contour integral \eqref{eq:omega-n} in this work, $e_n(\tau\lambda)$ has a similar property as
$\exp(\lambda t_n)$ for $Re(\lambda)\leq 0$ and $F_{\omega}(\lambda)$ may have weak singularity.
The optimal Talbot contour $z(\theta,N/t)$ derived
in \cite{Weideman06} works very well for the numerical inverse Laplace transform for a fixed $t$,
where the optimal Talbot contour is obtained by a numerical approach and the machine precision is obtained with $N=32$. We find that this optimal contour  $z(\theta,N/T_{\ell})$ still works well  for the contour integral in the current work for $t\in[B^{\ell-1},2B^{\ell}]$, see Figure \ref{eg31fig4}(b).
\end{remark}
}

\subsection{Error analysis}
In this subsection, we analyse the error of  Algorithm \ref{fast_conv}.
The error of the fast method depends only on the approximation $\hat{\omega}_{n}$
(see \eqref{omega-n-2}) to the contour integral \eqref{omega-n}. We have the following
error bound.
\begin{theorem}\label{thm3-1}
Let $\hat{\omega}_{n}$ be defined by \eqref{omega-n-2}, in which
${\lambda}_j^{(\ell)}$ and  ${w}_j^{(\ell)}$ are given by \eqref{weights-2},
and $F_{\omega}(\lambda)$ is defined by \eqref{BDF-F} or \eqref{GNGF-F}.
For  $t=n\tau$, $n\in[B^{\ell-1},2B^{\ell})$, and $n\geq b\mu_{\ell}t\geq 1$,
there exists  a positive constant  $C$  independent of $\tau,n$ and $N$ such that
\begin{eqnarray}
 |\hat{\omega}_{n}- {\omega}_{n}| \leq &&C\tau^{\alpha}
 \bigg[\frac{e^{a_0\mu_{\ell}t/2}}{e^{2d\pi/h}-1}+e^{(a_1-a_2\cosh(Nh))\mu_{\ell}t/2}\label{thm31}\\
 &&+e^{a_1\mu_{\ell}t/2}\left(1+\frac{b\mu_{\ell}t}
 {2(n-Q)}\cosh(Nh)\right)^{1-n+Q}\bigg],\nonumber
\end{eqnarray}
where   $Q$ is chosen such that
$(\tau\lambda)^{\alpha-1}\lambda^{\alpha}F_{\omega}(\lambda)/(1-\tau\lambda)^{Q+1}$ is bounded
for $\lambda(w)=\mu_{\ell}\left(1-\sin(\psi +iw)\right)+\sigma$,
$-d\leq \mathrm{Im}(w) \leq d$,
$0<\psi-d<\psi+d<\pi/2-\varphi$, $\psi,d>0$,
and $\varphi<\pi/2$.
\end{theorem}
\begin{proof}
We  follow the proof of Theorem 3 in \cite{LopLubPS2005} to
prove \eqref{thm31}, the detail is omitted here. See also Theorem 3.1 in
\cite{SchLopLub06}.
\end{proof}

Given a precision $\epsilon$, we can choose suitable parameters, such that
$|\hat{\omega}_{n}- {\omega}_{n}| \leq  C\tau \epsilon$,
see \cite{LopLubPS2005,SchLopLub06}.
Combining  \eqref{omega-n}, \eqref{omega-n-2}, and \eqref{eq:gauss-3} yields
\begin{equation}
|u^{(\ell)}_n-\widehat{u}^{(\ell)}_n|
=|\sum_{j=b^{(n)}_{\ell}}^{b^{(n)}_{\ell-1}-1}(\omega_{n-j}-\hat{\omega}_{n-j})u_j|
\leq C\tau (b^{(n)}_{\ell-1}-b^{(n)}_{\ell}) \|u\|_{\infty}\epsilon,
\end{equation}
which leads to
\begin{equation}\label{fast_conv_error}
|{}_FD^{(\alpha,n)}_{\tau,n_0}u-D^{(\alpha,n)}_{\tau}u|
\leq Ct_n \|u\|_{\infty}\epsilon.
\end{equation}
Denote by
\begin{equation}\label{fast-conv-2}
{}_FD^{(\alpha,n,m,\sigma)}_{\tau,n_0}u={}_FD^{(\alpha,n)}_{\tau,n_0}u
+ {\tau^{-\alpha}}\sum_{j=1}^{m}w^{(\alpha)}_{n,j}(u_j-u_0),
\end{equation}
where $w^{(\alpha)}_{n,j}$ is defined as in \eqref{s4:Dalf}.
Then from   \eqref{s4:eq-1} and \eqref{fast_conv_error},
the overall discretization error of  ${}_FD^{(\alpha,n,m,\sigma)}_{\tau,n_0}u$
is given by
\begin{equation}\label{eq:err}\begin{aligned}
|{}_FD^{(\alpha,n,m,\sigma)}_{\tau,n_0}u-k_{-\alpha}*u(t_n)|
=&|{}_FD^{(\alpha,n)}_{\tau,n_0}u-D^{(\alpha,n)}_{\tau}u
+ D^{(\alpha,n,m,\sigma)}_{\tau}u -k_{-\alpha}*u(t_n) |\\
\leq&|{}_FD^{(\alpha,n)}_{\tau,n_0}u-D^{(\alpha,n)}_{\tau}u|
+ |D^{(\alpha,n,m,\sigma)}_{\tau}u -k_{-\alpha}*u(t_n) |\\
\leq& C\left(\epsilon t_n \|u\|_{\infty}
+ \tau^pt_n^{\sigma_{m+1}-p-\alpha}+\tau^{\sigma_{m+1}+1}t_n^{-\alpha-1}\right).
\end{aligned}\end{equation}

The trapezoidal rule based on the optimal Talbot contour in \cite{Weideman06}  works well
(see \eqref{weights}), and needs  {fewer} quadrature points  to achieve the desired accuracy.


Denote  the relative pointwise error $e^{(T)}_n$ as
$$e^{(T)}_n=|{}_FD^{(\alpha,n)}_{\tau,n_0}u-D^{(\alpha,n,0,\sigma)}_{\tau}u|
/|D^{(\alpha,n,0,\sigma)}_{\tau}u|,$$
where $u(t)=t^2+t$, $D^{(\alpha,n,0,\sigma)}_{\tau}$ is defined by  \eqref{s4:Dalf}, and
${}_FD^{(\alpha,n)}_{\tau,n_0}$ is obtained from Algorithm  \ref{fast_conv}  based on
the Talbot quadrature (see \eqref{weights}).
We can similarly define the relative pointwise error $e^{(H)}_n$
based on the hyperbolic contour quadrature  \eqref{weights-2}.

Figure \ref{eg31fig4} displays the errors $e^{(H)}_n$ and $e^{(T)}_n$,
where we set $\alpha=0.5,\tau=0.01,B=5$, and $n_0=50$ in the computation
with the generating function defined by \eqref{gngf} and $p=2$.
We can see that the fast method based on both the hyperbolic quadrature
and the Talbot quadrature shows highly accurate
numerical approximations, and the Talbot quadrature
uses  {fewer} quadrature points than that of the hyperbolic contour  quadrature defined
by \eqref{weights-2}.  In \cite{SchLopLub06}, a strategy for
choosing the parameters for the hyperbolic contour  quadrature was proposed,
which may help to reduce the number of  quadrature points.

If the generating function \eqref{gngf} with different $p$ is applied,
we obtained similar results to those reported above,
which are not shown here. For the generating function defined by \eqref{fbdf},
$n_0$ needs to be chosen up to 200 for $p=6$ to ensure high accuracy.
One can verify that the present algorithm works well for $\alpha\in (-1,1)$,
these results are not shown here.

\begin{figure}[!h]
\begin{center}
\begin{minipage}{0.47\textwidth}\centering
\epsfig{figure=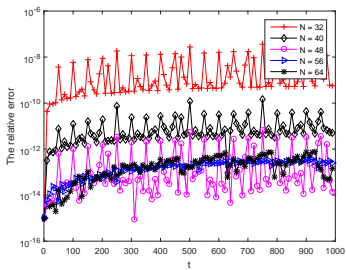,width=5cm}
\par {(a) The hyperbolic quadrature}
\end{minipage}
\begin{minipage}{0.47\textwidth}\centering
\epsfig{figure=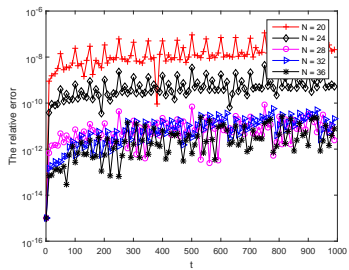,width=5cm}
\par {(b) The Talbot quadrature}
\end{minipage}
\end{center}
\caption{The relative pointwise errors  of the fast method based on the hyperbolic contour quadrature
and Talbot quadrature,  $\alpha=0.5,\tau=0.01,B=5$.\label{eg31fig4}}
\end{figure}

\section{Numerical examples}\label{sec:numerical}
In this section, two examples are presented to verify the effectiveness
of the present semi-implicit and fast method when it is applied to solve nonlinear FDEs.
All the algorithms are implemented using MATLAB 2017b,
run on a 3.40 GHz PC having 16GB RAM and Windows 7 operating system.
\begin{example}\label{s5-eg-1}
Consider the following scalar FODE
\begin{equation}\label{sec5:eq-1}
{}_{C}D^{\alpha}_{0,t}u(t) = - u(t) + f(u,t),{\quad}u(0)=u_0,{\quad}t\in(0, T],
\end{equation}
where $0<\alpha \leq 1$ and $f(u,t)$ is  a nonlinear function with respect $u$.
\end{example}

We apply the semi-implicit method \eqref{s4:IMEX} to solve \eqref{s5-eg-1},
where we set $\lambda=-1$ in  \eqref{s4:IMEX} for this example.
We also apply the fully implicit method \eqref{s4:IM} to solve \eqref{s5-eg-1}
for comparison,  where the Newton iteration method is used to solve
the corresponding nonlinear system.
When we say the fast method \eqref{s4:IMEX} or \eqref{s4:IM} is applied,
we mean that the discrete operator ${D}_{\tau}^{(\alpha,n,m,\sigma)}$ in
\eqref{s4:IMEX} or \eqref{s4:IM} is replaced by ${}_F{D}_{\tau,n_0}^{(\alpha,n,m,\sigma)}$,
where  ${}_F{D}_{\tau,n_0}^{(\alpha,n,m,\sigma)}$ is derived using Algorithm \ref{fast_conv}
(see \eqref{eq:gauss-3}).  We  always choose   the basis $B=5$, $n_0=50$,  and $N=32$
when the fast method based on the Talbot contour is applied.
For simplicity, we also set $m_u=m_f=m$ when  \eqref{s4:IMEX} or \eqref{s4:IM}  is applied.

The following three cases are considered in this example.
\begin{itemize}
  \item Case I: For  $f=-2u$ and $u_0=3$, the exact solution of \eqref{sec5:eq-1} is
  $u(t)=E_{\alpha}(-3t^{\alpha}),$
  where $E_{\alpha}(t)=\sum_{k=0}^{\infty}\frac{t^{k}}{\Gamma(k\alpha+1)}$ is the  Mittag--Leffler function  \cite{Pod-B99}.
  \item Case II: Let  $f=-u^2+g(t)$. Choose a suitable
  initial condition and $g(t)$ such that the exact solution of \eqref{sec5:eq-1} is
  $u(t)=2+t+{t^2}/{2}+{t^3}/{3}+{t^4}/{4}.$
  \item Case III: Let  $f=u(1-u^2)+2\cos(2\pi t)$ with
  the initial condition  taken as $u_0=1$.
\end{itemize}

The maximum error is defined by
$$\|e\|_{\infty}=\max_{0\leq n \leq {T}/\tau}\big|e_n\big|,  {\quad}e_n=u(t_n)-U_n,$$
where $U_n$ is either the numerical solution from the fast method or the direct method.



The purpose of Case I is to verify the effectiveness of the present semi-implicit and fast method
for non-smooth solutions.
We demonstrate that adding correction terms improves the accuracy of the numerical
solutions significantly, see    Tables  \ref{s5:tb1} and \ref{s5:tb2},
where the maximum relative error and the relative  error at $t=40$ for $\alpha=0.4$ are displayed.
We see that adding correction terms increases the overall accuracy and convergence rate,
readers can refer to \cite{DieFord06,Lub86,ZengZK17} for related results.
\begin{table}[!h]
\caption{{The maximum relative error $\|e\|_{\infty}/\|u\|_{\infty}$ of the semi-implicit method \eqref{s4:IMEX}
with fast convolution, Case I, $\sigma_k=k\alpha$,  $\alpha=0.4$, $\kappa=2$, $B=5$, $N=32$,
and ${T=40}$.}}\label{s5:tb1}
\centering\footnotesize
\begin{tabular}{|c|c|c|c|c|c|c|c|c|c|c|c|c|}
\hline
 $\tau$ & $m=0$ & Order& $m=1$ & Order& $m=2$ & Order & $m=3$   & Order\\
 \hline
$2^{-7} $ &7.6319e-2&      &7.0844e-4&      &8.2078e-5&      &6.4330e-5&      \\
$2^{-8} $ &6.4796e-2&0.2361&4.8320e-4&0.5520&4.3412e-5&0.9189&3.2473e-5&0.9862\\
$2^{-9} $ &5.3713e-2&0.2706&3.1526e-4&0.6161&2.1695e-5&1.0007&1.5317e-5&1.0841\\
$2^{-10}$ &4.3664e-2&0.2988&1.9873e-4&0.6657&1.0359e-5&1.0664&6.7300e-6&1.1865\\
$2^{-11}$ &3.4942e-2&0.3215&1.2208e-4&0.7030&4.7706e-6&1.1187&2.8040e-6&1.2631\\
\hline
\end{tabular}
\end{table}

\begin{table}[!h]
\caption{{The relative error $\left|e_n\right|/\|u\|_{\infty}$ of  the semi-implicit method \eqref{s4:IMEX}
with fast convolution at ${t=40}$, Case I, $\sigma_k=k\alpha$,  $\alpha=0.4$, $\kappa=2$,
$B=5$, and $N=32$.}}\label{s5:tb2}
\centering\footnotesize
\begin{tabular}{|c|c|c|c|c|c|c|c|c|c|c|c|c|}
\hline
 $\tau$ & $m=0$ & Order& $m=1$ & Order& $m=2$ & Order & $m=3$   & Order\\
 \hline
$2^{-7} $&1.8929e-6&      &8.8630e-8&      &3.4878e-8 &      &1.8337e-8 &      \\
$2^{-8} $&9.4643e-7&1.0000&3.1027e-8&1.5143&1.0813e-8 &1.6895&6.6615e-9 &1.4609\\
$2^{-9} $&4.7322e-7&1.0000&1.0765e-8&1.5272&3.2796e-9 &1.7212&2.2083e-9 &1.5929\\
$2^{-10}$&2.3661e-7&1.0000&3.7238e-9&1.5315&9.7775e-10&1.7460&6.8314e-10&1.6927\\
$2^{-11}$&1.1830e-7&1.0000&1.2906e-9&1.5287&2.8842e-10&1.7613&2.0452e-10&1.7400\\
\hline
\end{tabular}
\end{table}

We now fix the fractional order $\alpha=0.2$ and let the time
stepsize change. In  Figure \ref{eg51fig2} (a), we let $\kappa = 0$
and $\tau=1.4\times 10^{-3},1.5\times 10^{-3},1.59\times 10^{-3},
1.6\times 10^{-3},1.7\times 10^{-3}$ in the computations. We can also see from Table
\ref{s4:stability-interval} that the method \eqref{s4:IMEX}  is stable
if $\tau< 1.59 \times 10^{-3}$.
We  observe that the numerical solutions become  unstable when $\tau>1.59 \times 10^{-3}$.
We also observe from Figure \ref{eg51fig2} (b) that the
numerical solutions diverge for $\kappa=0.4$ when $\tau>1.1\times 10^{-2}$.
These results verify the theoretical findings  shown in Table \ref{s4:stability-interval}.

\begin{figure}[!h]
\begin{center}
\begin{minipage}{0.47\textwidth}\centering
\epsfig{figure=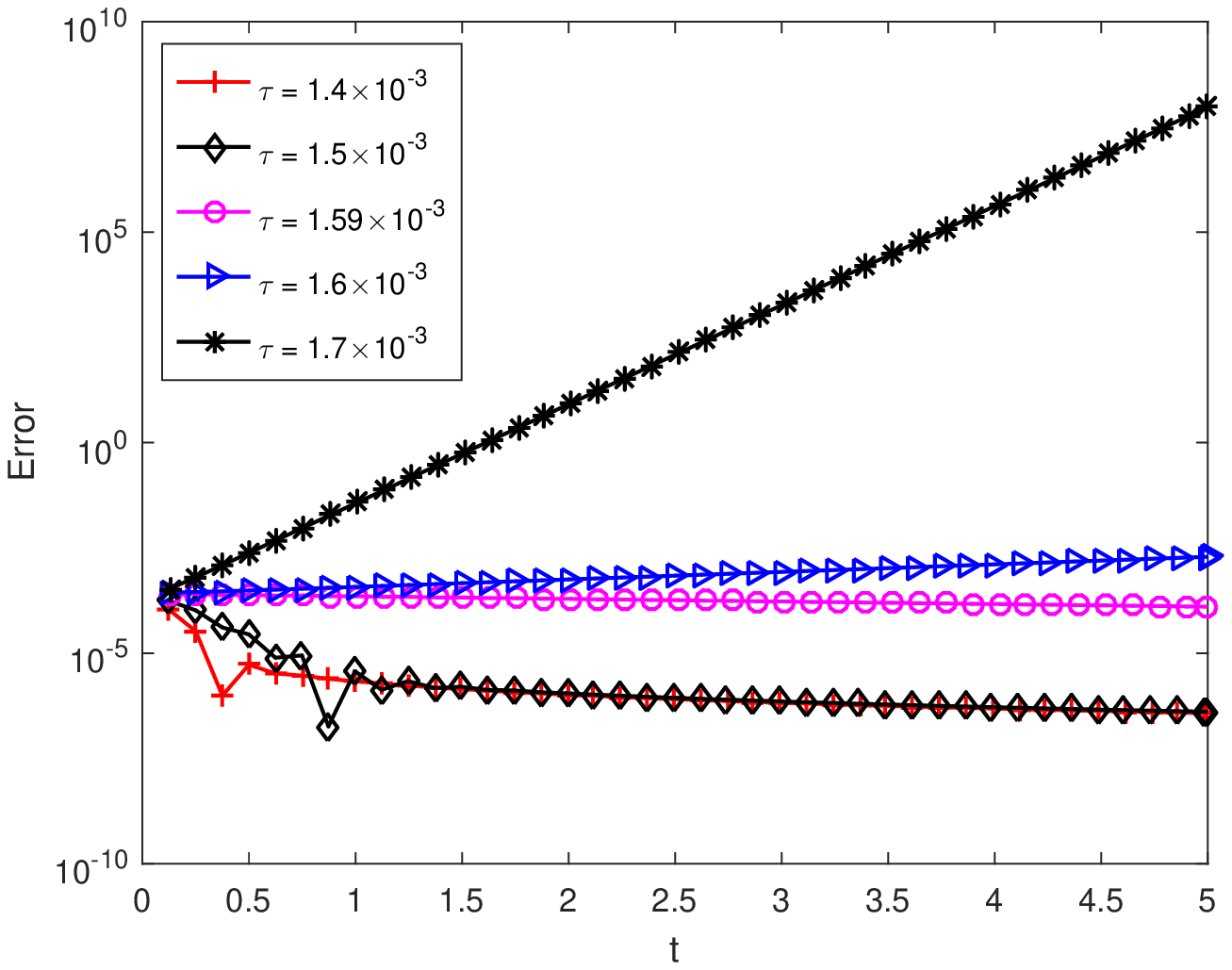,width=5cm} \par {(a) $\kappa=0$.}
\end{minipage}
\begin{minipage}{0.47\textwidth}\centering
\epsfig{figure=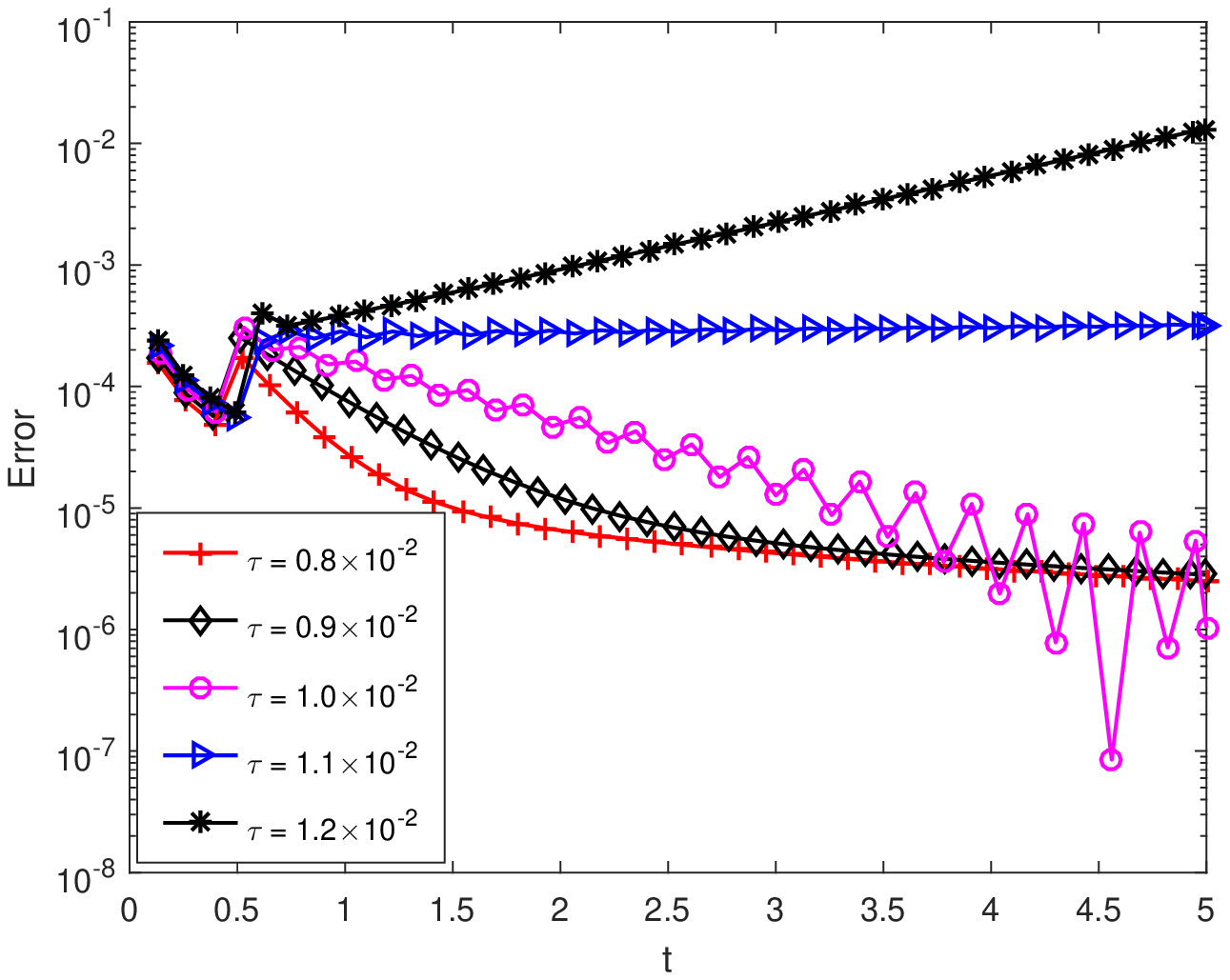,width=5cm} \par {(b) $\kappa=0.4$.}
\end{minipage}
\end{center}
\caption{The pointwise  errors of the semi-implicit method \eqref{s4:IMEX}
with different time stepsizes, Case I,
 $\alpha=0.2,m=1,B=5,N=32$.\label{eg51fig2}}
\end{figure}

For Case II, we solve a  nonlinear FODE with $f=-u^2+g(t)$ and $\px[u]f \in [-434\frac{5}{6},-4]$ for $t\in[0,5]$.
Here we give a simple   guideline for the choice of $\kappa$ in the method \eqref{s4:IMEX}.
From the linear stability analysis of the method \eqref{s4:IMEX}, we have that $\px[u]f$
plays a similar role  as $\rho$ does in  \eqref{s4:IMEX-3}, which can be obviously observed from the second semi-implicit method  \eqref{s4:IMEX2}.
Theorem \ref{thm:4-2} provides   a practical guideline for selecting $\kappa$ in real applications,
and the stability region of  \eqref{s4:IMEX} becomes larger as $\kappa$ increases, see Figure \ref{fig-stability-2}. From \eqref{thm:eq:4-2}, we can choose
$\kappa= \frac{1}{4}\max({-1-3\px[u]F})=325.875$ for the computations.
The relative absolute
errors $|e^n|/\|u\|_{\infty}$ at $t=5$ for different fractional orders are shown
in Table \ref{s5:tb3}. We can see that the present semi-implicit method
exhibits good stability and second-order accuracy for Case II.


\begin{table}[!h]
\caption{{The relative error $|e^n|/\|u\|_{\infty}$  of  the semi-implicit method \eqref{s4:IMEX}
with fast convolution at $t=5$, Case II, $\sigma_1=1$,  $m=1$,
$B=5,N=32$.}}\label{s5:tb3}
\centering\footnotesize
\begin{tabular}{|c|c|c|c|c|c|c|c|c|c|c|c|c|}
\hline
 $\tau$ & $\alpha=0.2$ & Order& $\alpha=0.5$ & Order& $\alpha=0.8$ & Order  \\
 \hline
$2^{-5}$&2.8685e-4&      &2.8700e-4&      &2.8694e-4&       \\
$2^{-6}$&7.2141e-5&1.9914&7.2180e-5&1.9914&7.2156e-5&1.9915 \\
$2^{-7}$&1.8089e-5&1.9957&1.8099e-5&1.9957&1.8092e-5&1.9958 \\
$2^{-8}$&4.5288e-6&1.9979&4.5313e-6&1.9979&4.5296e-6&1.9979 \\
$2^{-9}$&1.1330e-6&1.9989&1.1337e-6&1.9989&1.1332e-6&1.9989 \\
\hline
\end{tabular}
\end{table}


Next, we fix   the stepsize $\tau=1/256$ and set different
$\kappa=326,350,400,500,1000$ in the computation for
$\alpha=0.2$ and  $0.8$; the pointwise errors are shown
in Figure  \ref{eg51fig2-2}.
The results
shown in Figure \ref{eg51fig2-2} confirm the theoretical analysis displayed in Theorem \ref{thm:4-2}.
It can also be observed that very large $\kappa$ may negatively impact the accuracy of the
numerical solutions. One remedy is to use a smaller stepsize, such that $\kappa$ can be
chosen suitably large to ensure both stability and accuracy.
Another choice is to use high-order penalty terms $E^n_q(U)(q\geq 3)$ in \eqref{s4:IMEX},
but the unconditional stability of the derived method may not be guaranteed even for very large $\kappa$,
i.e., $q=3$.
Higher order stable semi-implicit methods will be studied in more detail in   our future work.

\begin{figure}[!h]
\begin{center}
\begin{minipage}{0.47\textwidth}\centering
\epsfig{figure=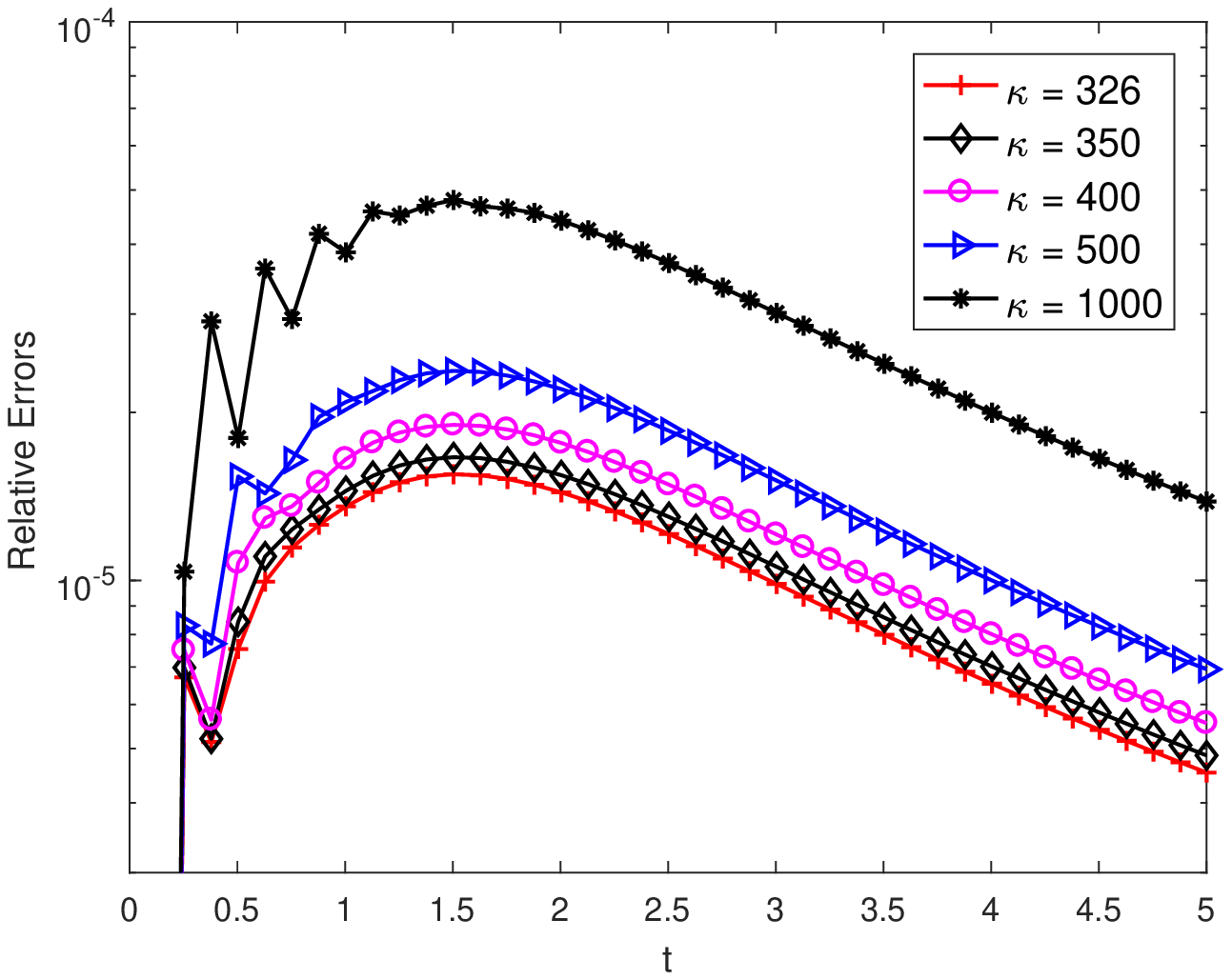,width=5cm}
\par {(a) $\alpha=0.2$.}
\end{minipage}
\begin{minipage}{0.47\textwidth}\centering
\epsfig{figure=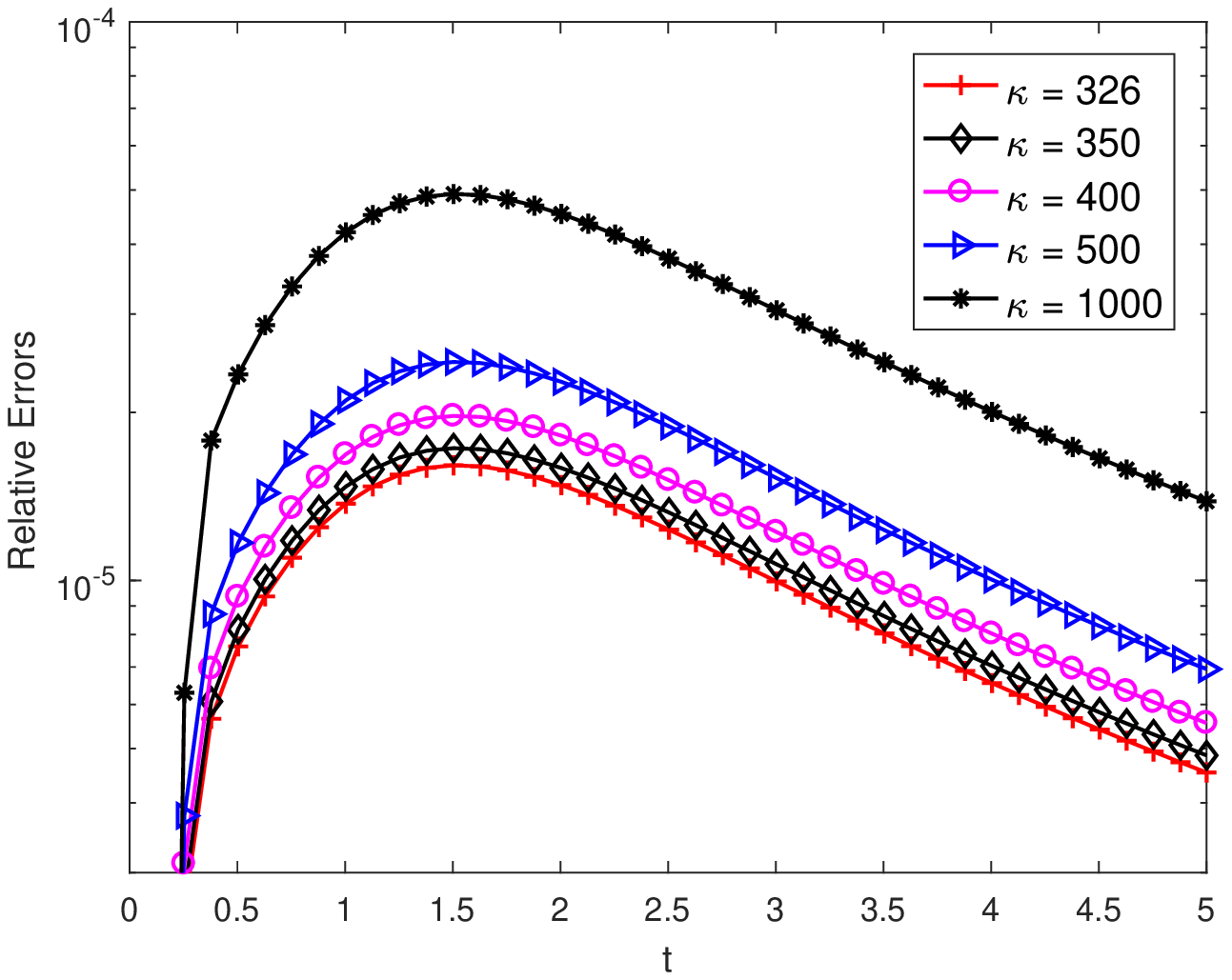,width=5cm}
\par {(b) $\alpha=0.8$.}
\end{minipage}
\end{center}
\caption{The relative pointwise  errors of the semi-implicit method \eqref{s4:IMEX}
with different $\kappa$, Case II,
 $m=1,B=5,N=32$.\label{eg51fig2-2}}
\end{figure}

Figures \ref{eg51fig3} (a)--(c) show  the numerical solutions
for Case III, where  the exact solution is not explicitly given.
We can see that the solutions
are bounded as theoretically expected \cite{WangXiao2015}.
Figure \ref{eg51fig3} (d) shows that the fast method is much  {more} efficient than
the direct method. We note that the fully implicit method with fast convolution
has almost a similar computational cost as the semi-implicit method with fast convolution.
Compared with the computational cost from the discrete convolution,
the computational cost from the Newton iteration method
to obtain the solution in the fully implicit method
can almost be ignored here, which is different from solving FPDEs as shown in the following example.

\begin{figure}[!h]
\begin{center}
\begin{minipage}{0.47\textwidth}\centering
\epsfig{figure=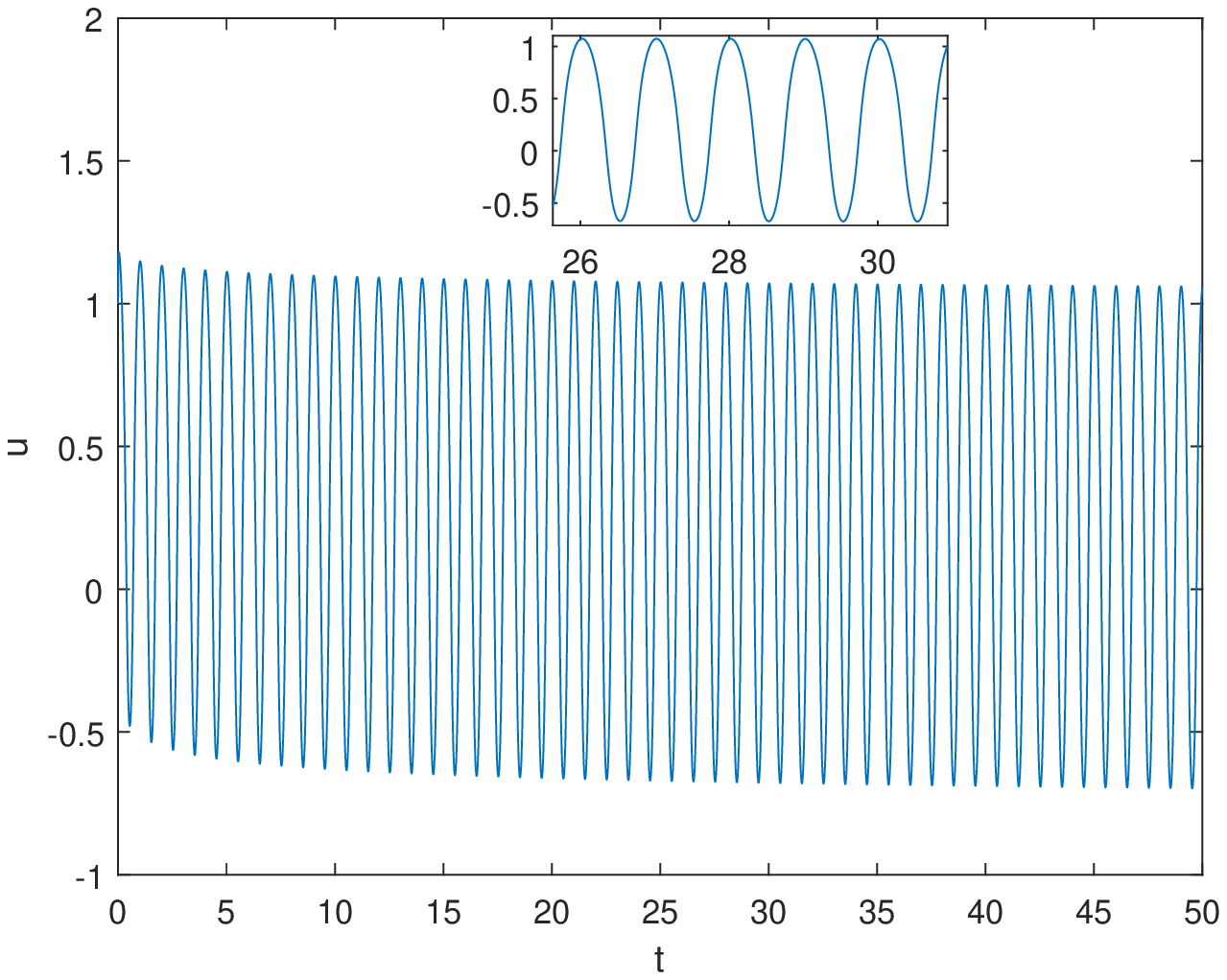,width=5cm} \par {(a) $\alpha=0.2$.}
\end{minipage}
\begin{minipage}{0.47\textwidth}\centering
\epsfig{figure=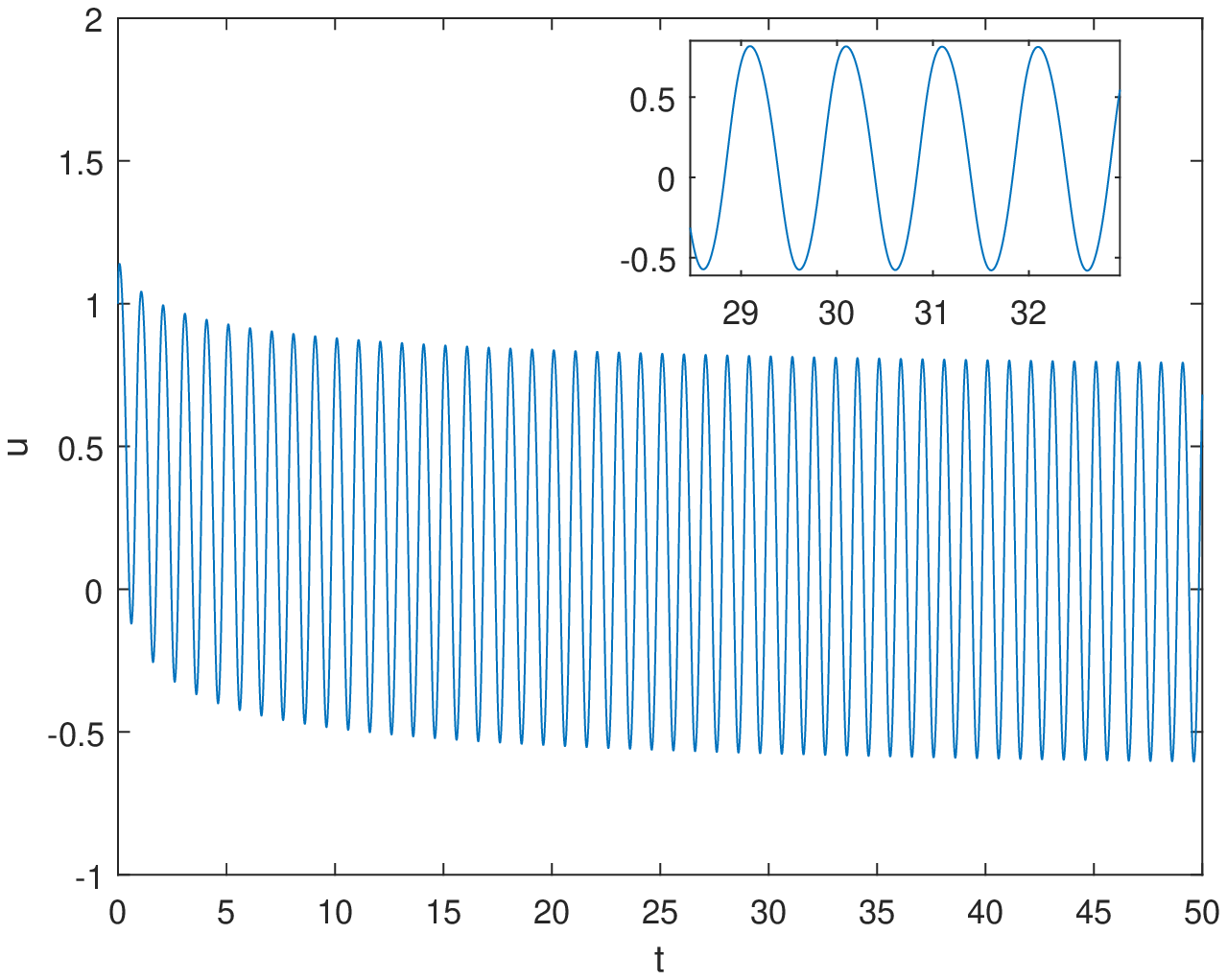,width=5cm} \par {(b) $\alpha=0.5$.}
\end{minipage}\\
\begin{minipage}{0.47\textwidth}\centering
\epsfig{figure=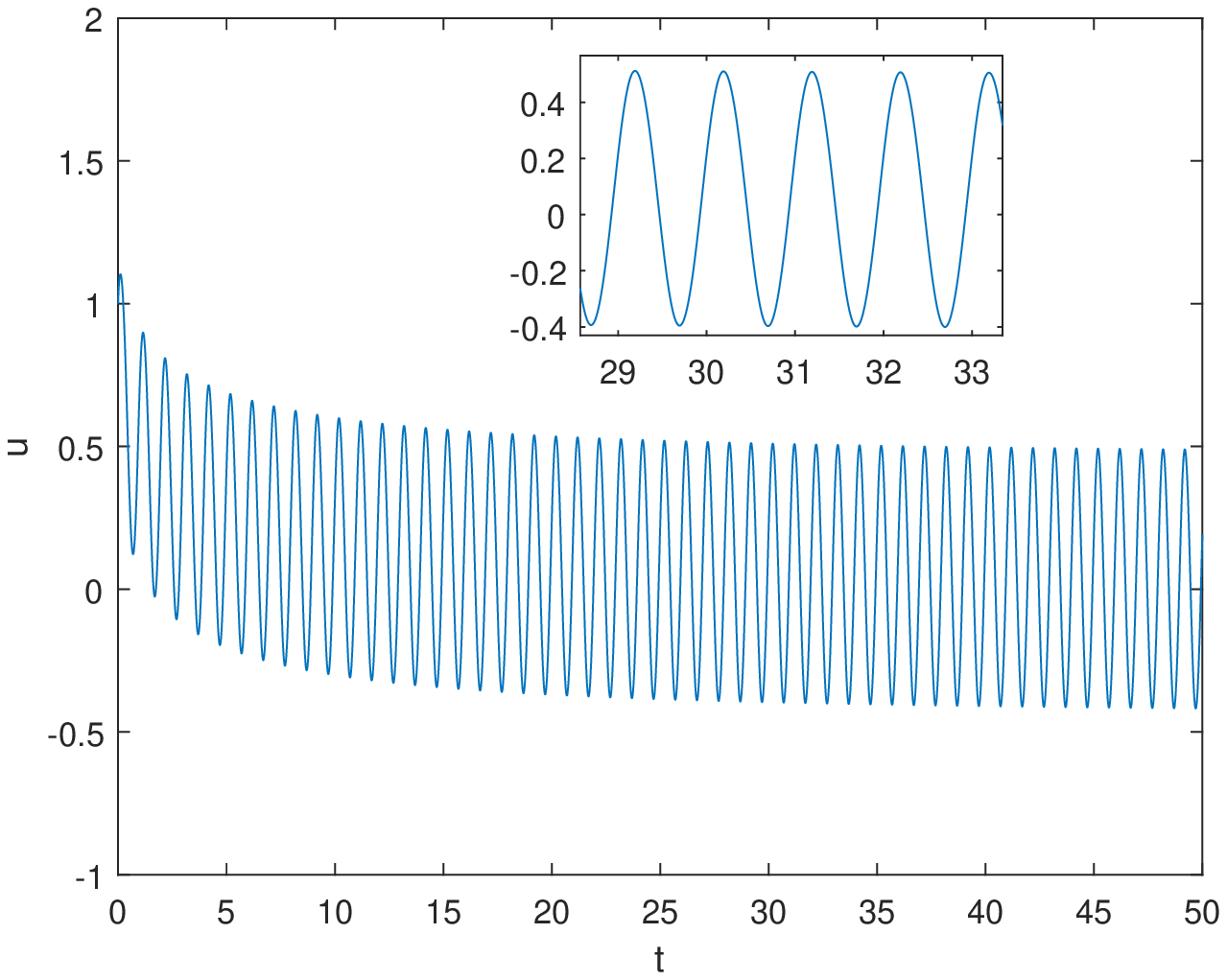,width=5cm} \par {(c) $\alpha=0.8$.}
\end{minipage}
\begin{minipage}{0.47\textwidth}\centering
\epsfig{figure=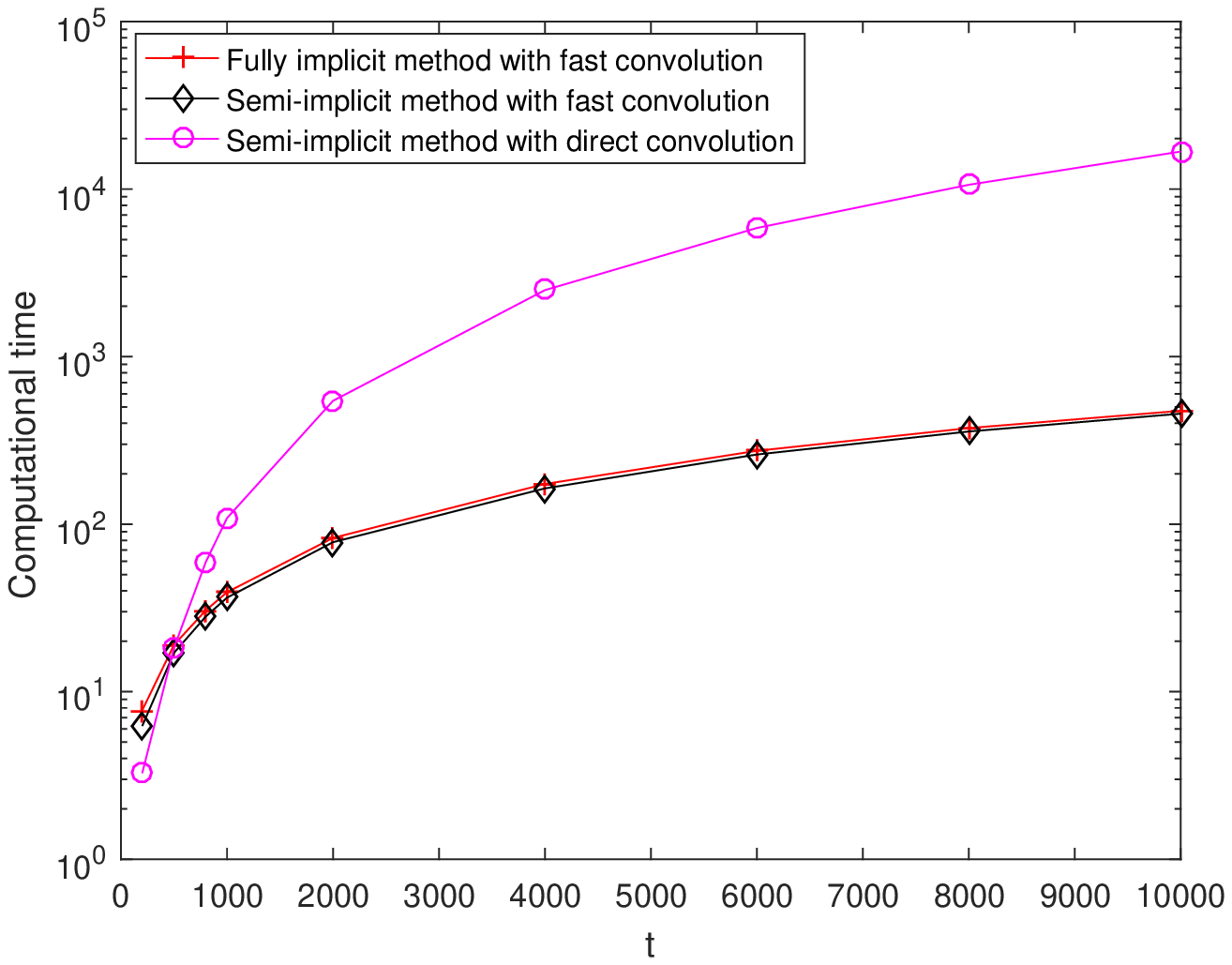,width=5cm} \par {(d) Computational time.}
\end{minipage}
\end{center}
\caption{Numerical solutions and the computational cost of different methods, Case II,
 $\kappa=3,\tau=0.005,m=1,B=5,N=32$.\label{eg51fig3}}
\end{figure}

{
For Case III,  we do not know the analytical solution and the
reference solutions are obtained using a smaller stepsize $\tau=2^{-13}$ with two correction terms.
Table \ref{s5:tb2-2} shows the relative error at $t=50$ with one correction
term, where second-order accuracy is observed for different fractional orders.
}
\begin{table}[!h]
\caption{{The relative error $\left|e_n\right|/\|u\|_{\infty}$ of  the semi-implicit method \eqref{s4:IMEX}
with fast convolution at ${t=50}$, Case III, $\sigma_k=k\alpha$,  $m=1$, $\kappa=3$,
$B=5$, and $N=32$.}}\label{s5:tb2-2}
\centering\footnotesize
\begin{tabular}{|c|c|c|c|c|c|c|c|c|c|c|c|c|}
\hline
 $\tau$ & $\alpha=0.1$ & Order&$\alpha=0.2$ & Order& $\alpha=0.5$ & Order& $\alpha=0.8$ & Order \\
 \hline
$2^{-5}$&1.1791e-2&      &1.2422e-2&      &1.8310e-2&      &1.3701e-2&      \\
$2^{-6}$&2.8763e-3&2.0354&3.4744e-3&1.8381&5.3884e-3&1.7647&2.8694e-3&2.2555\\
$2^{-7}$&7.2408e-4&1.9900&8.6976e-4&1.9981&1.3788e-3&1.9664&6.4843e-4&2.1458\\
$2^{-8}$&1.8166e-4&1.9949&2.1738e-4&2.0004&3.5033e-4&1.9767&1.5312e-4&2.0823\\
$2^{-9}$&4.5338e-5&2.0025&5.4153e-5&2.0051&8.8143e-5&1.9903&3.7019e-5&2.0483\\
\hline
\end{tabular}
\end{table}

\begin{example}\label{s5-eg-2}
Consider the following system of FPDEs
\begin{equation}\label{sec5:eq-2}\left\{\begin{aligned}
&{}_{C}D^{\alpha_1}_{0,t}u(t) = \mu_1 \Delta u(t) + f(u,v,x,y,t),\\
&{}_{C}D^{\alpha_2}_{0,t}w(t) = \mu_2 \Delta v(t) + g(u,v,x,y,t)
\end{aligned}\right.  \end{equation}
subject to the homogenous boundary conditions,
where $0< \alpha_1,\alpha_2\leq 1$, $\mu_1,\mu_2>0$, $u(t)=u(x,y,t)$,
 $v(t)=v(x,y,t)$, and $(x,y,t)\in (0,1)\times(0,1)\times(0,T]$.
The initial conditions are taken as $u(0)=u(x,y,0)=u_0(x,y)$ and
$v(0)=v(x,y,0)=v_0(x,y)$.
\end{example}

We focus on the three time-stepping methods for \eqref{sec5:eq-2}; the 2D
space is   discretized using the   standard second-order
finite volume method (FVM) based on a uniform grid.
The second-order generating function $\omega(2,\alpha,\tau,z)=\left(\frac{1-z}{\tau}\right)^{\alpha}
(1+\frac{\alpha}{2}-\frac{\alpha}{2}z)$ is applied, see \eqref{gngf}.
For simplicity, no correction terms are applied in this example.

We first make a slight modification of the semi-implicit time-stepping method \eqref{s4:IMEX}
to   \eqref{sec5:eq-2}, which yields the following semi-discrete method (see \eqref{s23:fode-2})
\begin{equation}\label{sec5:eq-3}\left\{\begin{aligned}
&D^{(\alpha_1,n,0,\sigma)}_{\tau}U = \mu_1 \Delta U_n+F_n - E^n_2(F)-\kappa_1  E^{n}_2(U),\\
&D^{(\alpha_2,n,0,\sigma)}_{\tau}V = \mu_2\Delta V_n+G_n - E^n_2(G)-\kappa_2  E^{n}_2(V),\\
\end{aligned}\right.  \end{equation}
where $F_n=f(U_n,V_n,x,y,t_n)$,   $G_n=g(U_n,V_n,x,y,t_n)$,
$D^{(\alpha,n,0,\sigma)}_{\tau}$ is defined by \eqref{s4:Dalf},
and  $E^n_2$ is defined by \eqref{s4:Ep}.

Applying  the FVM to the space discretization of each equation in \eqref{sec5:eq-3},
we obtain the fully discrete semi-implicit FVM with direct convolution.
If   $D^{(\alpha_k,n,0,\sigma)}_{\tau}$ in \eqref{sec5:eq-3} is replaced by
${}_FD^{(\alpha_k,n,0,\sigma)}_{n_0,\tau}$,
where ${}_FD^{(\alpha_k,n,0,\sigma)}_{n_0,\tau}$
is defined by  \eqref{fast-conv-2}, then the  fully discrete
semi-implicit FVM with fast convolution is derived.
The fully discrete implicit FVM with direct convolution can be derived by applying the time
discretization \eqref{s4:IM} to each equation of \eqref{sec5:eq-3} with space approximated
by the FVM. When the new fast method is applied, we always set $B=5,n_0=50,N=32$.
\begin{itemize}[leftmargin=*]
  \item Case I: Let $u_0=v_0=\sin(\pi x)\sin(\pi y)$, $f=-vu^2 +\hat{f}(x,y,t)$,
  and $g=-v^2u+\hat{g}(x,y,t)$.  Choose suitable  $\hat{f}$ and $\hat{g}$
  such that the exact solution  to \eqref{sec5:eq-2} is
  $$u = E_{\alpha_1}(-t^{\alpha_1})\sin(\pi x)\sin(\pi y),{\quad}
  v =E_{\alpha_2}(-t^{\alpha_2})\sin(\pi x)\sin(\pi y). $$
  \item  Case II: Let $u_0=x(1-x)y(1-y),v_0=\sin(\pi x)\sin(\pi y)$, $f=-u^2v$, and $g=-v^2u$.
\end{itemize}

Table \ref{s5:tb4} compares the efficiency and accuracy of the  implicit FVM
and the semi-implicit FVM for Case I, in which the nonlinear algebraic system
arising due to the coupling of the
implicit FVM is
solved by  fixed point iteration. Obviously, the semi-implicit method is faster
than the implicit method, but is a little less accurate than the implicit method, which
can be explained from the truncation error $R^n$ of the semi-implicit method defined by \eqref{s4:Rn},
i.e., $R^n=O(\tau^2t_n^{\sigma_{1}-2-\alpha})+O(\tau^{\sigma_{1}+1}t_n^{-\alpha-1})
+O(\tau^2t_n^{\delta_{1}-2})+O(\tau^2t_n^{\sigma_{1}-2})$.
The time truncation error of the implicit method  contains only the first two terms of $R^n$.
However, the truncation error of the semi-implicit method also contains
the third and fourth terms of $R^n$ that may play dominant roles here,  which leads to
a little less accurate numerical solution of the semi-implicit method.
Because of the symmetry, the same accuracy of the numerical solutions of $u$ and $v$ are observed.


\begin{table}[!h]
\caption{Comparison of the semi-implicit FVM and the
 implicit FVM at $t=2$,  Case I, $\mu_1=\mu_2=1$, $\alpha_1=\alpha_2=0.5$, $\kappa_1=\kappa_2=2$,
and $h=1/256$.}\label{s5:tb4}
\centering\footnotesize
 \begin{tabular}{|c|ccc|ccc|}
\hline
\multicolumn{1}{|c|}{ } &
\multicolumn{3}{|c|}{Semi-implicit method}&\multicolumn{3}{|c|}{Fully implicit method} \\\hline
$1/\tau$&$L^2$-error(u)&$L^2$-error(v) & Time(s) &$L^2$-error(u)&$L^2$-error(v) & Time(s)\\\hline
8  &8.8202e-4&8.8202e-4&0.7331&5.9450e-5&5.9450e-5&3.3517  \\
16 &2.2203e-4&2.2203e-4&1.5142&2.6788e-5&2.6788e-5&6.5332  \\
32 &5.6555e-5&5.6555e-5&3.4880&1.3454e-5&1.3454e-5&13.318  \\
48 &2.4377e-5&2.4377e-5&5.9117&9.7126e-6&9.7126e-6&20.602   \\
64 &1.2608e-5&1.2608e-5&8.7127&8.0249e-6&8.0249e-6&27.735   \\
\hline
\end{tabular}
\end{table}

We compare in Table \ref{s5:tb5} the semi-implicit direct method with  the semi-implicit fast method
for Case I.
Clearly, the two methods achieve almost similar numerical solutions. This can be explained from
the fact that the time discretization error of the fast method contains two parts, one part is the
same as that of the direct method,
the other part is from the quadrature to disctetize the contour integral,
which is independent of and also far smaller than the first part due to the sufficient number
of quadrature points,
see related results shown in Figure \ref{eg31fig4}(b). The fast method is much more efficient
than the direct method. ``Out of memory'' occured for the direct method when $t>1200$, while
the fast method works for $t\gg 1200$ and the computational cost increases almost linearly,
 see Table \ref{s5:tb5}.
Theoretically, the computational time of the direct method increases proportional to $n_T^2$, so
the computational time at $t=400$ is about $4\times 19207.5798$ seconds, but we did not
obtain the results within the expected time due to the memory problem.

\begin{table}[!h]
\caption{Comparison of the semi-implicit method with direct convolution
and fast convolution, Case I,   $\alpha_1=0.2,\alpha_2=0.8$,
$h=1/128,\tau=0.01$, and $\kappa_1=\kappa_2=2$.}\label{s5:tb5}
\centering\footnotesize
 \begin{tabular}{|c|ccc|ccr|}
\hline
\multicolumn{1}{|c|}{ } &
\multicolumn{3}{|c|}{Direct convolution}&\multicolumn{3}{|c|}{Fast convolution} \\\hline
$t$&$L^2$-error(u) &$L^2$-error(v) & Time(s) &$L^2$-error(u)&$L^2$-error(v)& Time(s)\\\hline
40  &7.0926e-6&2.5317e-7&1253.9869 &7.0926e-6&2.5317e-7&994.6982    \\
100 &6.4289e-6&1.4192e-7&4849.6792 &6.4289e-6&1.4192e-7&2618.0711   \\
200 &5.7959e-6&8.0380e-8&19207.5798&5.7959e-6&8.0380e-8&5463.2218   \\
300 &5.7810e-8&5.7810e-8&43056.7172&5.7810e-8&5.7810e-8&7954.6040   \\
400 &    -    &    -    &     -    &5.2029e-6&4.5797e-8&11330.6522   \\
1000&    -    &    -    &    -     &4.4865e-6&2.1880e-8&29548.2192   \\
2000&    -    &    -    &    -     &3.9966e-6&1.2539e-8&60915.4783   \\
\hline
\end{tabular}
\end{table}

For Case II, we do not have an explicit form of the analytical solution, we show numerical solutions in Figures \ref{eg51fig4} and \ref{eg51fig5}. Figures \ref{eg51fig4} (a) and (b) display, respectively,
the numerical solutions of $u$ and $v$ at $t=50$ for $(\alpha_1,\alpha_2)=(0.8,0.2)$.
Figures \ref{eg51fig4} (c) and (d) show  numerical solutions of $u$ and $v$  at different times with $y=0.5$,
we see that the numerical solutions decay as time evolves, see also Figure \ref{eg51fig5}.
For other fractional orders,  similar behavior
can be observed, see  Figures \ref{eg51fig5} (a) and (b) for $\alpha_1=\alpha_2=0.5$.

\begin{figure}[!h]
\begin{center}
\begin{minipage}{0.4\textwidth}\centering
\epsfig{figure=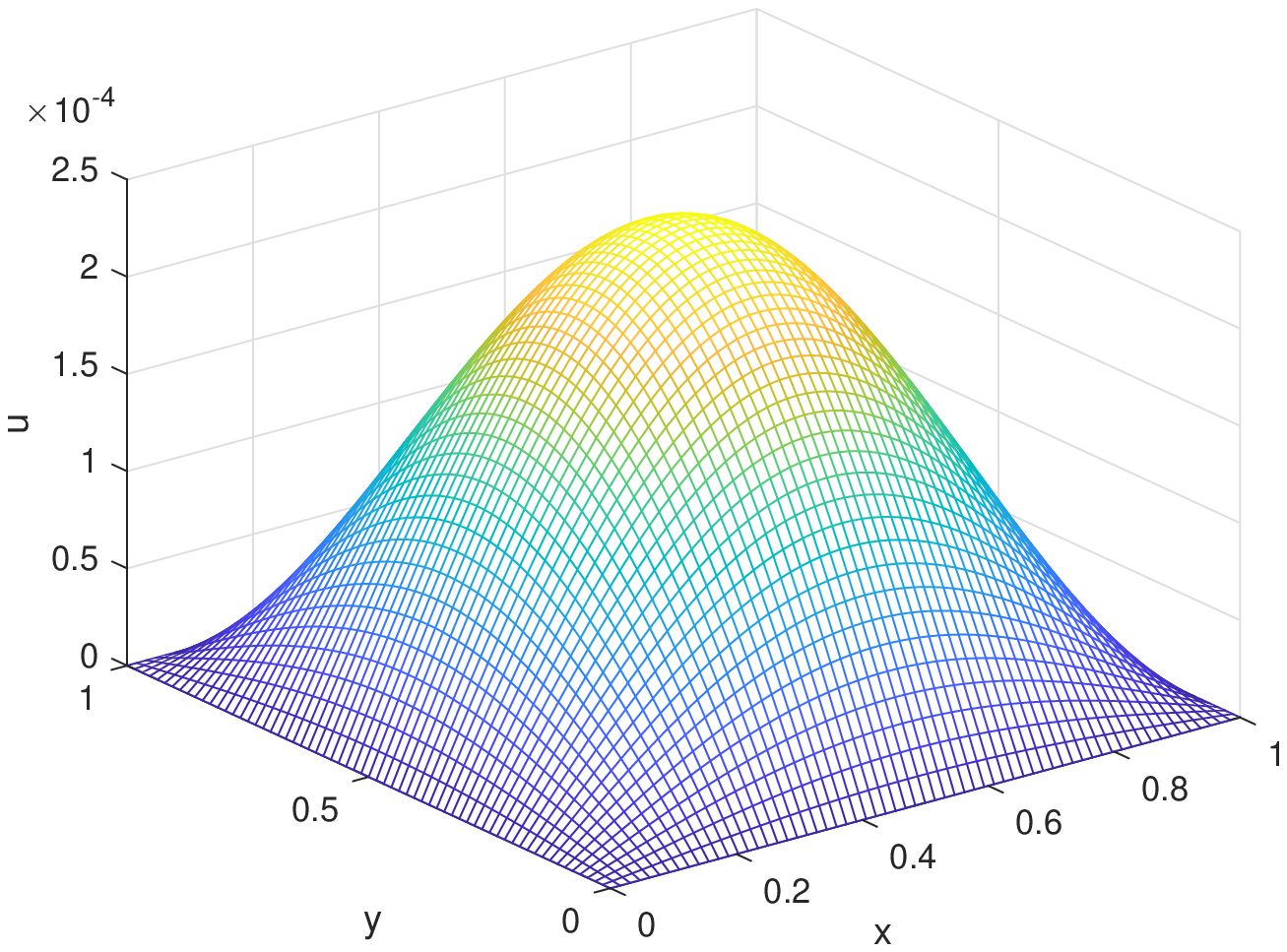,width=5cm} \par {(a) $t=50$.}
\end{minipage}
\begin{minipage}{0.4\textwidth}\centering
\epsfig{figure=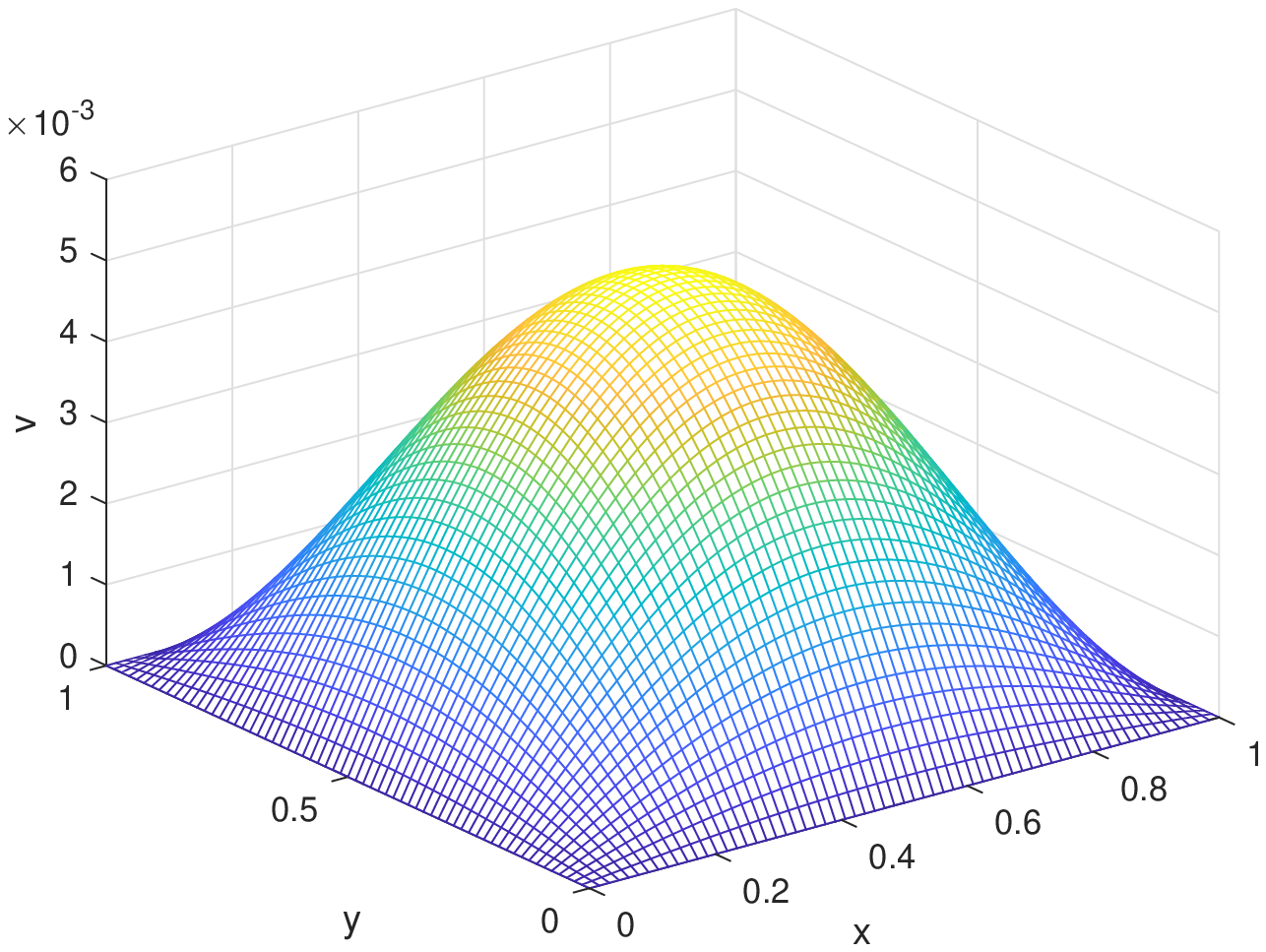,width=5cm} \par {(b) $t=50$}
\end{minipage}\\
\begin{minipage}{0.4\textwidth}\centering
\epsfig{figure=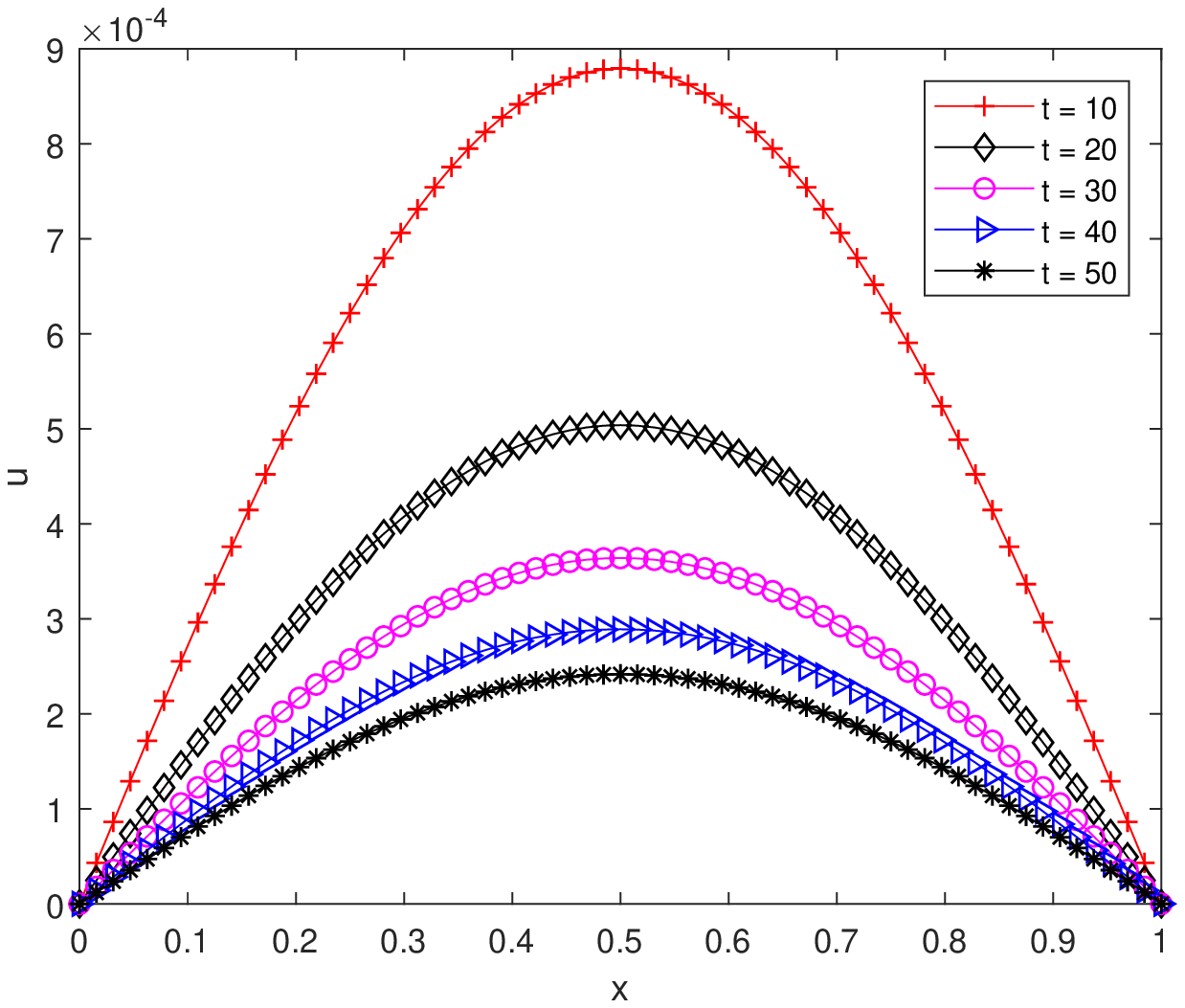,width=5cm} \par {(c) $y=0.5$.}
\end{minipage}
\begin{minipage}{0.4\textwidth}\centering
\epsfig{figure=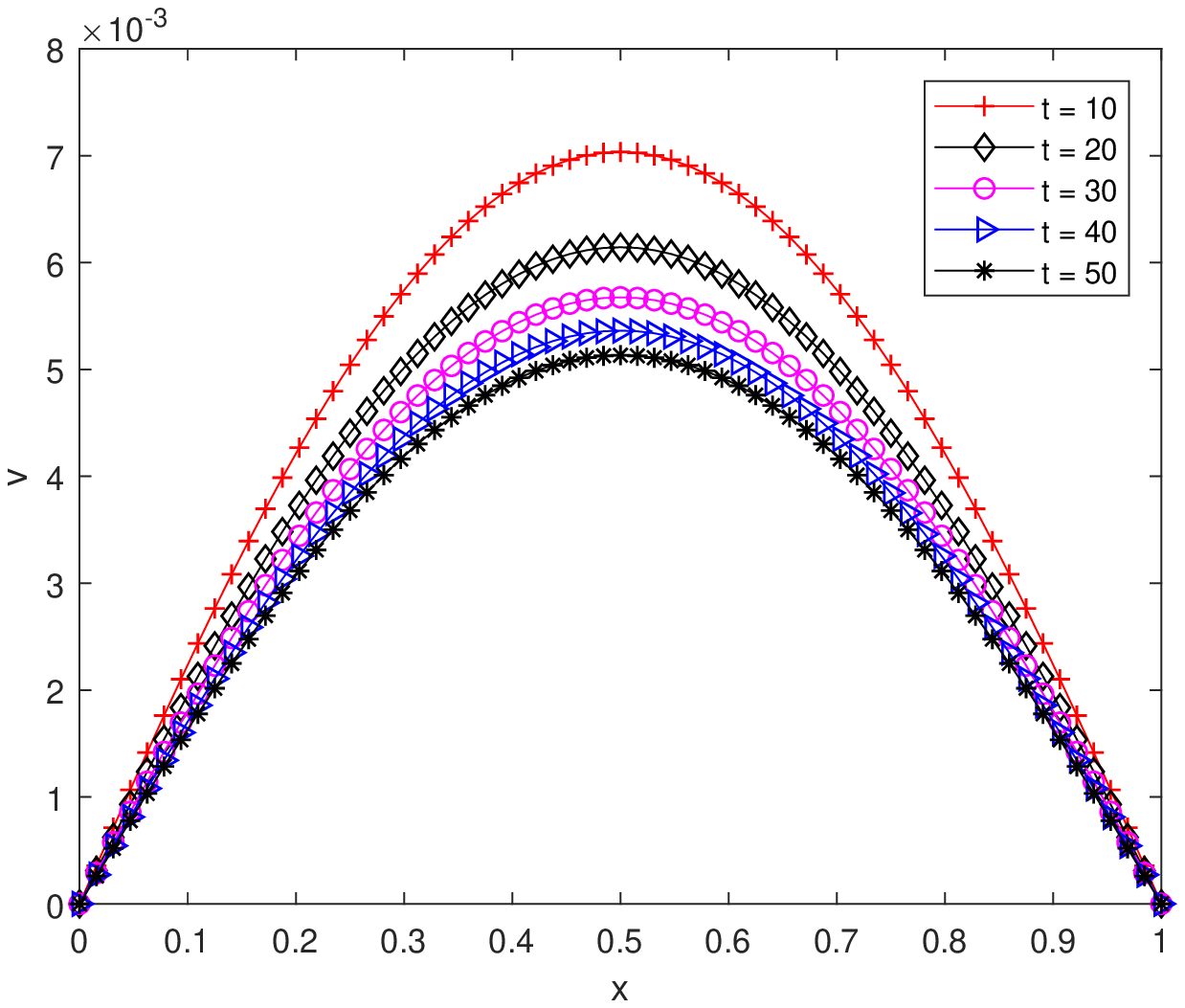,width=5cm} \par {(d) $y=0.5$.}
\end{minipage}
\end{center}
\caption{Numerical solutions  for Case II, $\alpha_1=0.8,\alpha_2=0.2$,
 $\kappa_1=\kappa_2=2,\tau=0.01,h=1/64,B=5,N=32$.\label{eg51fig4}}
\end{figure}

\begin{figure}[!h]
\begin{center}
\begin{minipage}{0.4\textwidth}\centering
\epsfig{figure=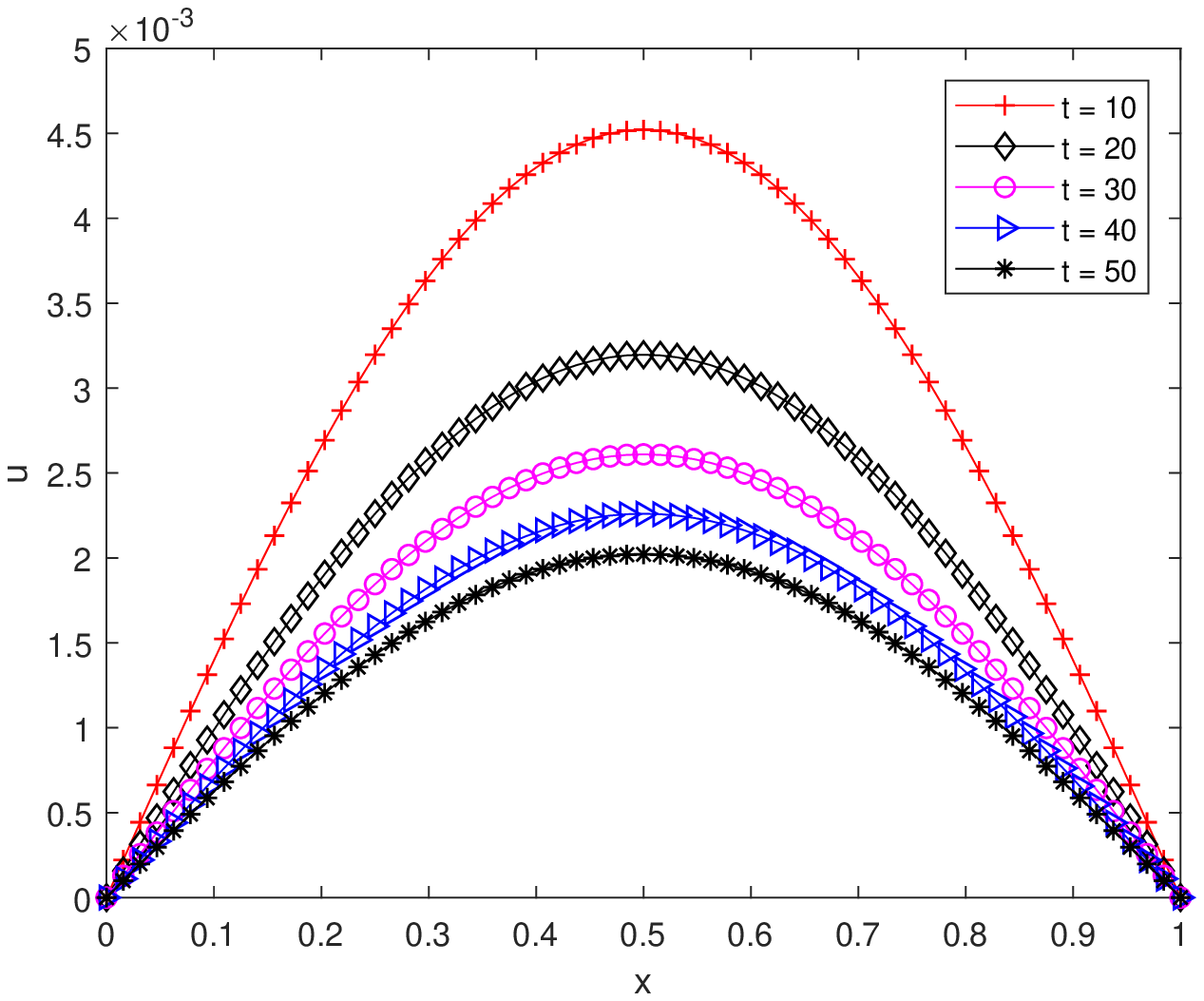,width=5cm} \par {(a) $y=0.5$.}
\end{minipage}
\begin{minipage}{0.4\textwidth}\centering
\epsfig{figure=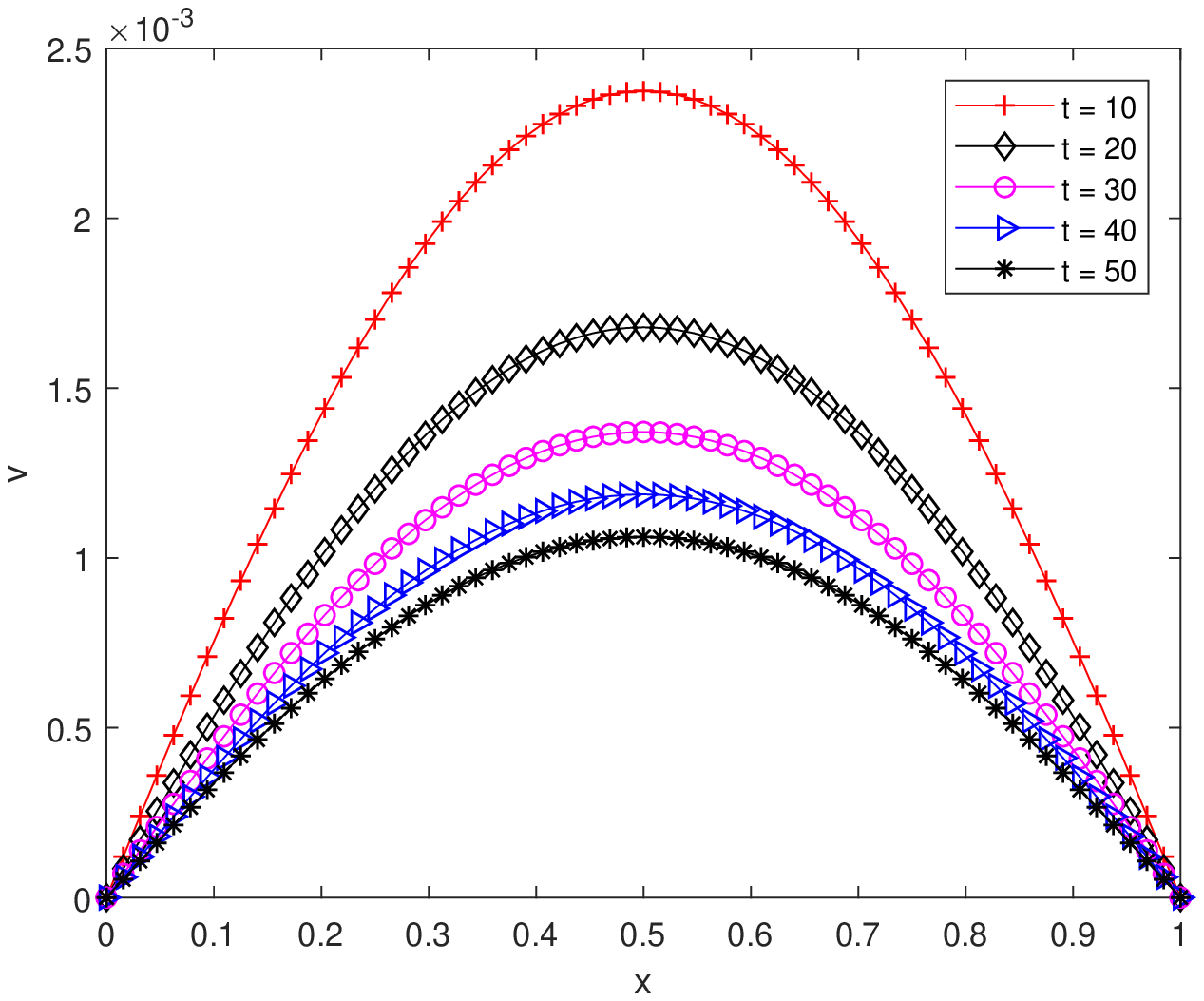,width=5cm} \par {(b) $y=0.5$.}
\end{minipage}
\end{center}
\caption{Numerical solutions  for Case II, $\alpha_1=\alpha_2=0.5$,
 $\kappa_1=\kappa_2=2,\tau=0.01,h=1/64,B=5,N=32$.\label{eg51fig5}}
\end{figure}

\section{Conclusion and discussion}\label{concl}
In this paper, we considered how to efficiently solve nonlinear time-fractional differential equations.
The nonlinearity is resolved by proposing a new class of semi-implicit time-stepping methods.
The stability of the new semi-implicit methods was investigated, and the stability interval was
theoretically given and verified numerically. In particular, several cases of the semi-implicit methods
are unconditionally stable by choosing   parameters that are explicitly given.
We also extend the semi-implicit methods for a scalar equation to a system of equations with
the stability condition given theoretically.

The nonlocality of the fractional operator leads to a discrete convolution,
which is costly by direct computation. This issue is resolved by the new fast
and memory-saving algorithm.  The new approach simplifies and extends the
method in \cite{SchLopLub06} to a wider class of discrete convolution methods for
approximating the fractional operator with coefficients generated from
the generating functions \cite{Lub86}. We prove that the error originating from the
fast calculation can be arbitrarily small and is independent of the truncation error
of the direct method. In the fast method, a series of ODEs with homogenous initial conditions
need to be solved
at each time step. These ODEs were solved by the multi-step method that corresponds
to the fractional multi-step method for the fractional operator in \cite{SchLopLub06}. However,
in the current work,
these ODEs are   solved using the \emph{backward Euler} method. If the $k$-step ($k>1$)
method is  applied, then additional errors may occur since   $k$ initial values are needed
to start the ODE solver, but only one initial value is known.

The semi-implicit methods in the present work achieve at most second-order accuracy, even though
a high-order time-stepping method for the fractional operator is applied.
In future work, we will consider how to construct uniformly stable high-order
semi-implicit methods for nonlinear time-fractional   differential equations.
Another computational issue is to seek a uniform quadrature with high accuracy
to discretize the integral contour \eqref{Flambda}, which will make it easier to  implement
the present fast method similarly to those given in \cite{BafHes17b,JiangZZZ16,JingLi10}.

\appendix
{
\section{Proofs of Theorems \ref{thm:4-1} and \ref{thm:4-1-b}}\label{appenix-A}
Proof of Theorem \ref{thm:4-1}.
\begin{proof}
Denote by $U(z) = \sum_{n=0}^{\infty}U_nz^n$, $|z|\leq 1$.
We have from \eqref{s4:IMEX-2}
\begin{equation*}\begin{aligned}\label{s4:IMEX-4}
\sum_{n=2}^{\infty}\sum_{j=0}^{n}\omega^{(\alpha)}_{n-j}(U_j-U_0)z^n
= (\lambda -\kappa)\tau^{\alpha} \sum_{n=2}^{\infty}U_nz^n
+ (\rho + \kappa)\tau^{\alpha}\sum_{n=2}^{\infty} (2U_{n-1}-U_{n-2})z^n,
\end{aligned}\end{equation*}
which leads to
\begin{equation}\begin{aligned}\label{s4:IMEX-5}
&\omega(p,\alpha,1,z)U(z)-\omega^{(\alpha)}_0U_0
- (\omega^{(\alpha)}_0U_1 + \omega^{(\alpha)}_1U_0)z
-U_0\sum_{n=2}^{\infty}\sum_{j=0}^{n}\omega^{(\alpha)}_{j}z^n\\
=& (\lambda -\kappa)\tau^{\alpha} (U(z)-U_0-U_1z)
+ (\rho + \kappa)\tau^{\alpha}(2z(U(z)-U_0)-z^2U(z)),
\end{aligned}\end{equation}
where we have used the following property
$$\omega(p,\alpha,1,z)U(z)=\sum_{n=0}^{\infty}\sum_{j=0}^{n}\omega^{(\alpha)}_{n-j}U_jz^n.$$

Rewrite \eqref{s4:IMEX-5} into the following form
\begin{equation}\begin{aligned}\label{s4:IMEX-6}
 U(z) = & \frac{H(z)}{\omega(p,\alpha,1,z) - \tau^{\alpha}(\lambda +\rho)
+ \tau^{\alpha}(\rho + \kappa)(1-z)^2},
\end{aligned}\end{equation}
where $H(z)=\sum_{n=0}^{\infty}H_nz^n$. According to  Eq. (3.3) and
Lemma 3.5 in \cite{Lub86}, one has
$\sum_{j=0}^{n}\omega^{(\alpha)}_{j}=\frac{n^{-\alpha}}{\Gamma(1-\alpha)}+O(n^{-\alpha-1})$,
which yields $H_n = O(n^{-\alpha})$. On the other hand, we have
$\omega(p,\alpha,1,z) - \tau^{\alpha}(\lambda +\rho)
- \tau^{\alpha}(\rho + \kappa)(1-z)^2=\sum_{n=0}^{\infty}Q_nz^n=Q(z)$ with
$Q_n=O(n^{-\alpha-1})$, where $\omega^{(\alpha)}_{n}=O(n^{-\alpha-1})$ has been used  \cite{Lub86}.
According to \cite{Lub86b}, $U_n\to 0$ as $n\to \infty$ if $Q(z)\neq0,|z|\leq 1$,
i.e., the method is stable if $\tau^{\alpha} \neq \frac{\omega(p,\alpha,1,z)}{(\lambda +\rho)
-(\rho + \kappa)(1-z)^2},|z|\leq 1$. The proof is complete.
\end{proof}

Proof of Theorem \ref{thm:4-1-b}.
\begin{proof}
Lemma 3.5 and Eq. (3.3) in \cite{Lub86} yield
$$\sum_{j=0}^{n}\omega^{(\alpha)}_{n-j}j^{\sigma_k}=\frac{n^{\sigma_k-\alpha}}{\Gamma(1+\sigma_k-\alpha)}
+O(n^{\sigma_k-\alpha-p}) + O(n^{-\alpha-1}),$$
which leads to
$$\sum_{j=1}^{m}w^{(\alpha)}_{n,j}j^{\sigma_k}
= O(n^{\sigma_k-\alpha-p})+O(n^{-\alpha-1}),{\quad}1\leq k \leq m.$$
The above linear system implies  that the starting weights $w_{n,j}^{(\alpha)}$ in \eqref{s4:Dalf} satisfy
$$w_{n,j}^{(\alpha)}=O(n^{\sigma_1-\alpha-p}) + O(n^{\sigma_2-\alpha-p})+\cdots
+O(n^{\sigma_m-\alpha-p})+ O(n^{-\alpha-1}).$$
By the same reasoning, we can obtain that the starting weights $w_{n,j}^{(f)}$ and $w_{n,j}^{(u)}$ used
in \eqref{s4:extrapolation-2nd} satisfy
\begin{eqnarray*}
w_{n,j}^{(u)}&=&O(n^{\sigma_1-q}) + O(n^{\sigma_2-q})+\cdots+O(n^{\sigma_{m_u}-q}),\\
w_{n,j}^{(f)}&=& O(n^{\delta_1-q}) + O(n^{\delta_2-q})+\cdots+O(n^{\delta_{m_f}-q}).
\end{eqnarray*}

Denote  $m_{\max}=\max\{m,m_u,m_f\}+1$ and
$$U(z)=\sum_{n=m_{\max}}^{\infty}U_nz^n.$$
For the method \eqref{s4:IMEX} with correction terms, we can similarly derive
\begin{equation}\begin{aligned}\label{B:IMEX-6}
 U(z) = & \frac{H(z)+\hat{H}(z)}{\omega(p,\alpha,1,z) - \tau^{\alpha}(\lambda +\rho)
+ \tau^{\alpha}(\rho + \kappa)(1-z)^2},
\end{aligned}\end{equation}
where $H(z)$ is similarly derived as in \eqref{s4:IMEX-6} with $H_n\to 0$ as $n\to\infty$,
$\hat{H}(z)=\sum_{n=0}^{\infty}\hat{H}_nz^n$,
 $\hat{H}_n=0, 0\leq n \leq m_{\max}$, and
$$\hat{H}_n=-\sum_{j=1}^{m}w^{(\alpha)}_{n,j}(U_j-U_0)
+\rho\tau^{\alpha}\sum_{j=1}^{m_f}w^{(f)}_{n,j}(U_j-U_0)
+\kappa\tau^{\alpha}\sum_{j=1}^{m_u}w^{(f)}_{n,j}(U_j-U_0)$$
for $n\geq m_{\max}$. Obviously,
$\hat{H}_n \to 0$ as $n\to \infty$ if
$${\sigma_k-\alpha-p}<0,{\quad}\sigma_{m_u}-q<0,{\quad}\sigma_{m_f}-q<0.$$
The above inequalities yield Theorem \ref{thm:4-1-b}, which ends the proof.
\end{proof}
}
\section{Proof of Theorem \ref{thm:4-2}}\label{appenix-B}
We prove Theorem \ref{thm:4-2} for  $p=q=2$ with $\omega(2,\alpha,1,z)=\tau^{\alpha}\omega(2,\alpha,\tau,z)
=(1-z)^{\alpha}(1+\frac{\alpha}{2}-\frac{\alpha}{2}z)$,
other cases can be similarly proved.
The following notations are used to simplify the proof, i.e.,
\begin{equation}
\begin{aligned}
&D=\{(x,y)|x^2+y^2<1\}, {\quad}\px[]D=\{(x,y)|x^2+y^2=1\},\\
&D_U=\{(x,y)|x^2+y^2<1,0<y<1\}, {\quad}\px[]D_U= \px[]D^{(1)}_U\cup \px[]D^{(2)}_U,
\end{aligned}\end{equation}
where $\px[]D^{(1)}_U=\{(x,y)|-1\leq x \leq 1, y = 0\},
\px[]D^{(2)}_U=\{(x,y)|x^2+y^2=1,|x|<1,0<y<1\}$.
Denote $\bar{D}=D\cup \px[]D$ and $\bar{D}_U=D_U\cup \px[]D_U$.
Then, $D_{U}$ is the interior of the upper semi-circular domain with
boundary $\px[]{D}_U$. We can similarly define the lower semi-circular domain $\bar{D}_L$.
The proof of Theorem  \ref{thm:4-2} is equivalent to proving that
$f(z)=\frac{\omega(2,\alpha,1,z)}{\lambda-\kappa+(\rho+\kappa)(2z-z^2)}$
is not positive real for $z \in \bar{D}$.
We  just need to prove that $f(z)$ is not positive real  for $z\in \bar{D}_U$,
which   also holds for  $z\in \bar{D}_L$ by the same reasoning.

\begin{proof}
For simplicity, we denote
$W(x,y)=(1-z)^{\alpha}(1+\frac{\alpha}{2}-\frac{\alpha}{2}z)$,
and $V(x,y)=\lambda-\kappa+(\rho+\kappa)(2z-z^2)$, where $z=x+iy=r\exp(i\theta)$, $i^2=-1$,
$\theta\in[-\pi,\pi],0\leq r \leq 1$.
Denote $W(x,y)=W_1(x,y) + iW_2(x,y)$ and
$V(x,y)=V_1(x,y) + iV_2(x,y)$, where
$W_1,W_2,V_1,V_2$ are real.
Then we have
\begin{equation}\label{eq:A1-1}\left\{\begin{aligned}
&W_1(x,y)\geq0,{\qquad\qquad~~}|x|^2+|y|^2\leq 1,\\
&W_2(\cos\theta,\sin\theta)\geq 0,{\qquad}\theta\in[-\pi,0],\\
&W_2(\cos\theta,\sin\theta)\leq0,{\qquad}\theta\in[0,\pi].
\end{aligned}\right.\end{equation}
The first inequality can be found in \cite{WangVong14}.
Let $\phi(\theta)=\Big(2\sin\frac{\theta}{2}\Big)^{\alpha}$. Then we have
$W_2(\cos\theta,\sin\theta) =\phi(-\theta)
\Big[\Big(1+\frac{\alpha}{2}-\frac{\alpha}{2}\cos\theta\Big)\sin\frac{\alpha}{2}(\theta+\pi)
-\frac{\alpha}{2}\cos\frac{\alpha}{2}(\theta+\pi)\sin(\theta)\Big]>0$,
$\theta\in(-\pi,0)$. The third inequality in \eqref{eq:A1-1} can be similarly
derived from $W_2(\cos\theta,\sin\theta)=\phi(\theta)
\Big[\Big(1+\frac{\alpha}{2}-\frac{\alpha}{2}\cos\theta\Big)\sin\frac{\alpha}{2}(\theta-\pi)
-\frac{\alpha}{2}\cos\frac{\alpha}{2}(\theta-\pi)\sin\theta\Big]<0, \theta\in(0,\pi)$.

For $V_1$ and $V_2$, we have the following results.
\begin{equation}\label{eq:A1-2}\left\{\begin{aligned}
&V_1(x,y)<0,{\qquad\qquad }
\frac{\lambda-3\rho}{4}<\kappa<-2\lambda-3\rho,\quad|x|^2+|y|^2\leq 1,\\
&V_2(\cos\theta,\sin\theta)\leq0,{\quad}\kappa>-\rho,\quad\theta\in[-\pi,0],\\
&V_2(\cos\theta,\sin\theta)\geq0,{\quad}\kappa>-\rho,\quad\theta\in[0,\pi].
\end{aligned}\right.\end{equation}
The first inequality in \eqref{eq:A1-2} can be derived by checking
$\widehat{V}_1(\pm 1)<0$ and $\widehat{V}_1(1/2)<0$,
where $\widehat{V}_1(\cos\theta)=V_1(\cos\theta,\sin\theta)
=\lambda-\kappa +(\rho+\kappa)(2\cos\theta-2\cos^2\theta+1).$
The second and third inequalities in \eqref{eq:A1-2} can be easily deduced from
$V_2(\cos\theta,\sin\theta) =(\rho+\kappa)(2\sin\theta-\sin2\theta)
=2(\rho+\kappa)\sin\theta(1-\cos\theta)$.

Next, we prove that $f(z)=\frac{\omega(2,\alpha,1,z)}{\psi_2(z)}=\frac{W(x,y)}{V(x,y)}$
cannot be positive real on $\bar{D}_U$.
\begin{itemize}[leftmargin=*]
  \item {Step 1)} Assume that $\frac{\lambda-3\rho}{4}<\kappa<-2\lambda-3\rho$ and
  there exists a positive constant $c>0$,
  such that $f(z)=\frac{W_1+iW_2}{V_1+iV_2}=c$. Then we have
  $W_1-cV_1=0$. Eqs.  \eqref{eq:A1-1} and \eqref{eq:A1-2} imply
  $W_1-cV_1>0$ for  all $|z|\leq 1$,  which yields a contradiction.
  Therefore, $f(z)$ cannot be a positive real number when
 $\frac{\lambda-3\rho}{4}<\kappa<-2\lambda-3\rho$.

  \item Step 2)   Since $\omega(2,\alpha,1,z)$ is  analytic  for $|z|<1$ and
  continuous for $|z|=1$,
  the  imaginary part $W_2(x,y)$ of $\omega(2,\alpha,1,z)$
  is a harmonic function in the interior of the unit circular domain $\bar{D}$
  and is continuous on the boundary  $\px[]{D}$ of $\bar{D}$.
  It implies that $W_2(x,y)$ (or $V_2(x,y)$) is also analytic in $D_U$ and continuous on
  $\px[]D_U$.
  Therefore,  $W_2$ (or $V_2$) attains
  its maximum and  minimum value on the boundary $\bar{D}_U$.
  From \eqref{eq:A1-1} and \eqref{eq:A1-2}, we have $W_2(x,y)\geq0$
  and $V_2(x,y)\leq0$ for $(x,y)\in \bar{D}_U$ if $\kappa>-\rho$.
  For any $(x,y)\in \bar{D}_U\setminus \px[]D^{(1)}_U$,
  we assume that there exists a positive number $c$ such that
  $f(z)=\frac{W_1+iW_2}{V_1+iV_2}=c$, which leads to
  $W_2-cV_2=0$. On the other hand,
  we have  $V_2(x,y)=(\rho+\kappa)2r\sin\theta(1-r\cos\theta)>0$  for
  $(x,y)=(r\cos\theta,r\sin\theta)\in \bar{D}_U\setminus \px[]D^{(1)}_U$,
  which yields $0=W_2-cV_2<0$. Therefore, $f(z)$ cannot be positive real for
  $z=(x,y)\in  \bar{D}_U\setminus \px[]D^{(1)}_U$.
  If $(x,y)\in \px[]D^{(1)}_U$, then $f(z)$ is real
  and $f(z)=\frac{W_1}{V_1}\leq0$ due to the fact $W_1\geq0$ and $V_1<0$.
  Therefore, $f(z)$ is not positive real on $\bar{D}_U$ if $\kappa>-\rho$.
\end{itemize}

From Steps 1) and 2), we prove that $f(z)$ is not positive real on $\bar{D}_U$
if $\kappa>\frac{\lambda-3\rho}{4}$.
We can similarly prove that $f(z)$ is not positive real on $\bar{D}_L$ $f(z)$.
Hence, $f(z)$ is not positive real
for $|z|\leq 1$, which implies
the stability region $\mathbb{S}$ defined \eqref{s4:stability-p} contains
the interval $(0,\infty)$. The proof is completed.
\end{proof}

\section*{Acknowledgments}
The authors wish to thank the referees for their constructive
comments and suggestions, which greatly improve
the quality of this paper.
%
\def\cprime{$'$} \def\cprime{$'$} \def\cprime{$'$} \def\cprime{$'$}

\end{document}